\documentclass[12pt]{article}

\usepackage{amsmath}
\usepackage{amssymb}
\usepackage{amsthm}
\usepackage{hyperref}
\usepackage{enumerate}
\usepackage{graphicx}
\usepackage{color}
\usepackage{xcolor}
\usepackage{mathtools}
\usepackage{float}
\usepackage{tikz}
\usetikzlibrary{arrows}
\usepackage{bm}
\usepackage{mathrsfs}
\usepackage{wasysym}
\usepackage[title]{appendix}

\newtheorem{thm}{Theorem}[section]

\newtheorem{lem}[thm]{Lemma}
\newtheorem{conj}[thm]{Conjecture}
\newtheorem{assumption}[thm]{Assumption}

\newtheorem{pr}[thm]{Proposition}

\newtheorem{definition}[thm]{Definition}

\newtheorem{example}[thm]{Example}
\newenvironment{exmp}{\begin{example}\rm}{\end{example}}
\newtheorem{remark}[thm]{Remark}
\newenvironment{rem}{\begin{remark}\rm}{\end{remark}}
\newtheorem{algorithm}[thm]{Algorithm}

\newtheorem{observation}[thm]{Observation}

\newtheorem{claim}[thm]{Claim}

\newcommand{\vertiii}[1]{{\left\vert\kern-0.25ex\left\vert\kern-0.25ex\left\vert #1
    \right\vert\kern-0.25ex\right\vert\kern-0.25ex\right\vert}}

\title{Markov Capacity for Factor Codes with an Unambiguous Symbol}

\author{Guangyue Han\textsuperscript{1}, Brian Marcus\textsuperscript{2} and Chengyu Wu\textsuperscript{3}\\[2ex]
\textsuperscript{1}The University of Hong Kong, \textit{ghan@hku.hk}\\
\textsuperscript{2}The University of British Columbia, \textit{marcus@math.ubc.ca}\\
\textsuperscript{3}The University of British Columbia, \textit{wuchengyu0228@gmail.com}}

\date{{\normalsize \today}}
\begin{document}\maketitle\thispagestyle{empty}

\begin{abstract}
In this paper, we first give a necessary and sufficient condition for a factor code with an unambiguous symbol to admit a subshift of finite type restricted to which it is one-to-one and onto. We then give a necessary and sufficient condition for the standard factor code on a spoke graph to admit a subshift of finite type restricted to which it is finite-to-one and onto. We also conjecture that for such a code, the finite-to-one and onto property is equivalent to the existence of a stationary Markov chain that achieves the capacity of the corresponding deterministic channel.
\end{abstract}

\section{Introduction}
\label{intro}

Shifts of finite type (SFT), and more generally sofic shifts, are spaces of bi-infinite sequences that play a prominent role in symbolic dynamics. Of particular interest are factor codes (onto sliding block codes) from one such space to another, as they represent ways
of encoding blocks in the domain space into blocks in the range space.  However, typically such maps are badly many-to-one.  So, it would be
useful to know when one can restrict to a subspace of the domain such that the code
is still onto and one-to-one or finite-to-one.  Consider the following properties.
Given an irreducible SFT $X$, a sofic shift $Y$, and
a factor code, $\phi: X \rightarrow Y$,
\begin{itemize}
\item[P1:]
There exists an SFT $Z\subset X$ such that $\phi\vert_Z$ is a conjugacy
onto $Y$.
\item[P2:]
There exists an SFT $Z\subset X$ such that $\phi\vert_Z$ is finite-to-one and
onto $Y$.
\item[P3:]
There exists a stationary Markov measure $\nu$ on $X$ s.t.
$\phi^*(\nu) = \mu_0$,  the unique measure of maximal entropy (mme for short) on $Y$.
\end{itemize}
We are interested in finding checkable, necessary and sufficient conditions for each of these properties and in determining relationships among these properties.
Clearly, P1 implies P2 and P2 implies P3 because, given P2,
any mme $\nu$ on
$Z$ satisfies P3 (see Proposition~\ref{Fto1}).

A factor code $\phi:X \rightarrow Y$ can be viewed as an
input-constrained, deterministic, but typically lossy, channel in the
information theoretic sense: an input $x$ determines a channel output $y = \phi(x)$. Our interest in P3 stems from the fact that it is equivalent to the condition that the Markov capacity achieves the capacity of this channel, i.e., there is an input Markov measure on $X$ that achieves capacity (See Sections~\ref{InfoThy} and \ref{FactorChannel} for more details).

Since $Y$ is the image of an irreducible shift space, it must be irreducible, and it follows
that  $\mu_0$ is indeed unique and fully supported on $Y$. However,  we do not require $\nu$ to be fully
supported on $X$.

For P1,  there are certainly some necessary conditions; for instance if $Y$ has a fixed
point, then $X$ must have a fixed point, and $Y$ must be an SFT.

We consider the special class of factor codes with an unambiguous symbol. This means that the alphabet of $Y$ is $\{0,1\}$ and in the block code $\Phi$ that generates $\phi$,  there is exactly one block $u$ s.t. $\Phi(u) = 1$.  In Theorem~\ref{one-to-one_SFT},
we characterize, for this class, all such $\phi$ for which there exists a shift space $Z\subset X$ s.t.
$\phi\vert_Z$ is a conjugacy onto $Y$ and show that such a $Z$ must necessarily be an SFT, i.e., P1.
In Theorem~\ref{domain_full_one-to-one}
we give a refined version of this result when $X$
is the full 2-shift.

For P2, we recall from a counterexample [MPW84, pp. 287-289] that P2 is not always satisfied.
We consider a subclass of factor codes with an unambiguous symbol,
motivated by that counterexample, called standard factor codes on spoke graphs
(for the definition, see Section \ref{Sec_Spoke}). In Theorem~\ref{spoke_three_eq}, for this subclass,
we characterize all such $\phi$ satisfying P2, and we show that for any $\phi$ in this subclass, P2 is equivalent to the existence of an SFT $Z \subset X$, such that $\phi|_Z$ is almost invertible and onto $Y$.

The same counterexample in [MPW84, pp. 287-289] shows that for standard factor codes on spoke graphs, P3 is not always satisfied.

We conjecture that for standard factor codes on spoke graphs, P3 and P2 are equivalent, i.e., if
there exists a stationary Markov measure $\nu$ on $X$ s.t.
$\phi^*(\nu) = \mu_0$, then  there exists an SFT $Z\subset X$ such that $\phi\vert_Z$ is finite-to-one and
onto $Y$; if true, then for this class, the same characterization for P2  holds
for P3.  In Proposition~\ref{partialconverse}, we prove
this in several special cases.  The proof combines the Chinese Remainder Theorem and a dominance condition.

We note that P3 is related to the property that a factor code from an
irreducible SFT to an irreducible SFT is Markovian, although in that case one assumes
that such $\nu$ is fully supported [BT84], [BP11].

It was
shown in [MPW84, Proposition 3.2] that P2 always holds if we relax SFT $Z$ to sofic $Z$.
Similarly, it was
shown in [MPW84, Corollary 3.3] that if we relax stationary Markov $\nu$ to stationary hidden Markov $\nu$, then P3 always holds.

We point the reader to a related paper which considers factor codes $\phi:X \rightarrow Y$ as deterministic channels and
for a given factor code $\phi$, characterizes those subshifts, of entropy strictly less than that of $Y$, that can be faithfully encoded through $\phi$ [Mac22].

The remainder of this paper is organized as follows. In Section~\ref{SymDyn}, we give
brief background on symbolic dynamics, focusing on SFTs, sofic shifts and factor codes.
In Section~\ref{InfoThy}, we describe a motivating problem from information theory.
In Section~\ref{FactorChannel}
we describe factor codes as special channels in information theory (as was done in
[MPW84]).
We introduce in Section~\ref{unambig} the class of
factor codes with an unambiguous symbol and for this class
consider
P1 in Section \ref{P1_for_unam}.  In Section \ref{Sec_Spoke} we introduce the subclass of standard factor codes on spoke
graphs and
consider P2 for this subclass in Section \ref{P2_for_spoke}.  In Section~\ref{Conjecture}, we consider P3 for this subclass and prove Proposition~\ref{partialconverse}.
Finally, in Section \ref{two_cycles}, we discuss standard factor codes on
another class of graphs.

{\em Acknowledgements:}  We thank Mike Boyle, Felipe Garc\'ia-Ramos, Sophie MacDonald, Tom Meyerovitch, Ronny Roth, and Paul Siegel for helpful
discussions.

\section{Notation and Brief Background from Symbolic Dynamics}
\label{SymDyn}


We introduce in this section some basic terms and facts in symbolic dynamics. For more details, see [LM95].

Let $\mathcal{A}$ be a finite alphabet. The {\em full $\mathcal{A}$-shift}, denoted by $X_{\mathcal{A}}$, is the collection of all bi-infinite sequences over $\mathcal{A}$. When $\mathcal{A}=\{0,1, \cdots, n-1\}$, the full shift is called the {\em full $n$-shift} and will be denoted by $X_{[n]}$. For any point $x=\cdots x_{-1} x_0 x_1 \cdots\in X_{\mathcal{A}}$, we use $x_i$ to denote the $i$-th coordinate of $x$ and $x_{[i,j]}$ to denote the block $x_ix_{i+1} \cdots x_j$.
For a block $x_1\cdots x_m$, we use $(x_1 \cdots x_m)^k$ to denote its $k$-concatenation
and $(x_1 \cdots x_m)^\infty$ to mean its infinite concatenation. The shift map $\sigma$ on $X_{\mathcal{A}}$ is defined by $(\sigma(x))_i=x_{i+1}$ for any $x\in X_{\mathcal{A}}$. A subset of $X_{\mathcal{A}}$ is a {\em shift space} if it is compact and is invariant under $\sigma$. For any positive integer $m$, we use $\mathcal{B}_m(X)$ to denote the set of all allowed blocks of length $m$ in a shift space $X$, and $\mathcal{B}(X) := \cup_{n} \mathcal{B}_n (X)$ is called the {\em language} of $X$. The {\em $N$-th higher block shift} of $X$ is the image $\beta_N(X)$ in the full shift over $\mathcal{A}^N$ where $\beta_N: X \rightarrow(\mathcal{A}^N)^{\mathbb{Z}}$ is defined by $(\beta_N(x))_i= x_{[i,i+N-1]}$ for any $x\in X$. A shift space $X$ is {\em irreducible} if for any $u, v\in \mathcal{B}(X)$, there is a $w\in \mathcal{B}(X)$ such that $uwv\in \mathcal{B}(X)$.

Let $\mathcal{A}_1$, $\mathcal{A}_2$ be two alphabets, $s,t$ be two fixed integers and let $X$ be a shift space over $\mathcal{A}_1$.
The map $\phi: X \rightarrow {\mathcal{A}_2}^{\mathbb{Z}}$ defined by $\phi(x)_i=\phi(x_{[i-s, i+t]})$ for any $i$ is called a {\em sliding block code with anticipation $t$ and memory $s$}. A sliding block code $\phi: X \rightarrow Y$ is {\em finite-to-one} if there is an integer $M$ such that $\vert\phi^{-1}(y) \vert\leq M$ for every $y\in Y$, and it is {\em one-to-one} when $M=1$. Moreover, the sliding block code $\phi: X\rightarrow Y$ is a {\em factor code} if it is onto, in which case $Y$ will be called the {\em factor of $X$}, and $\phi$ is a {\em conjugacy} if it is one-to-one and onto.

A {\em point diamond} for $\phi$ is a pair of distinct points in $X$ that differ in finitely many coordinates and have the same image under $\phi$. If $X$ is irreducible, then $\phi$ is finite-to-one iff it has no point diamonds [LM95, Theorem 8.1.16].

Let $G$ be a directed graph with no multiple edges. For a path $\gamma$ in $G$, $V(\gamma)$ denotes the sequence of vertices of $\gamma$ and $\vert \gamma \vert$ is the length, i.e., the number of edges, of $\gamma$ (for example, for $\gamma=e_1 e_2 \cdots e_n$, $V(\gamma)=I(e_1) I(e_2) \cdots I(e_n) T(e_n)$ and $\vert \gamma\vert=n$, where for any $i$, $I(e_i)$ and $T(e_i)$ denote the initial vertex and the terminal vertex of $e_i$, respectively). We use $\mathcal{V}(G)$ to denote the vertex set of $G$ and $\widehat{X_G}$ to denote the {\em vertex shift induced by $G$}. That is, the shift space whose points are sequences of vertices of bi-infinite paths in $G$. Let $\Phi: \mathcal{V}(G) \rightarrow \mathcal{A}$ be a labelling of vertices of $G$ over a finite alphabet $\mathcal{A}$. A {\em graph diamond of $\Phi$} is a pair of distinct paths in $G$ that have the same initial vertex, terminal vertex and label.  It is well-known that, assuming $G$ is irreducible, the factor code generated by $\Phi$ is finite-to-one iff $\Phi$ has no graph diamonds [LM95, Section 8.1].

A shift space $X$ can be expressed as $X = X_{\mathcal{F}}$ where $\mathcal{F}$ is a {\em forbidden set}, a list of forbidden words such that $x \in X$ iff $x$ contains no element of $\mathcal{F}$. The choice of the forbidden set of $X$ is in general not unique. When $X=X_\mathcal{F}$ for some finite set $\mathcal{F}$, $X$ is called a {\em shift of finite type} (SFT). An SFT $X$ is called $M$-step (or has {\em memory} $M$) if $X=X_{\mathcal{F}}$ for a collection $\mathcal{F}$ of $(M+1)$-blocks. A vertex shift is always a $1$-step SFT and conversely, by lifting to its $(M+1)$-th higher block shift, an $M$-step SFT can always be represented as the vertex shift of a graph. A shift space $Y$ is {\em sofic} if there exist an SFT $X$ and a sliding block code $\phi$ such that $\phi(X)=Y$. Clearly, SFTs must be sofic.

There is a general definition of the degree of a factor code on any subshift, see [LM95, Definition 9.1.2.]. For our purposes, we focus only on the following equivalent definition of the degree of a 1-block finite-to-one factor code $\phi: X\rightarrow Y$ where $X$ is an irreducible $M$-step SFT $X$: Let $N:=\max\{1, M\}$. The {\em degree} of $\phi$ is defined as the minimum over all blocks $w=w_1 w_2 \cdots w_{\vert w \vert}$ in $Y$ and all $1\leq i\leq \vert w\vert-N+1$ the number of distinct $N$-blocks in $X$ that we see beginning at coordinate $i$ among all the pre-images of $w$ [LM95, Proposition 9.1.12]. A word $w$ that achieves the minimum above with some coordinate $i$ is called a {\em magic word}, and
 the subblock $w_{i}w_{i+1}\cdots w_{i+N-1}$ is called the corresponding {\em magic block}.
 
A factor code $\phi$ is {\em almost invertible} if its degree is $1$. While an almost invertible code need not be finite-to-one, on an irreducible SFT it must be finite-to-one [LM95, Proposition 9.2.2].

The {\em topological entropy} of a shift space $X$ is
$$h_{top} (X) := \lim_{m\rightarrow \infty} \frac{1}{m} \log \vert \mathcal{B}_m (X) \vert.$$
For a probability measure $\mu$ on $X$, let $h(\mu)$ denote its measure theoretic entropy. By the variational principle [Wal82, Theorem 8.6],
\begin{align} \label{variational}
h_{top} (X) = \sup_\mu \{h(\mu): \mu \mbox{ is a shift-invariant Borel probability measure on $X$} \}.
\end{align}
The {\em measure of maximal entropy (mme)} $\mu_0$ of $X$ is a probability measure on $X$ such that the supremum in (\ref{variational}) is achieved.

Given $S\subset \mathbb{Z}_{\geq 0}$, an {\em $S$-gap shift} $X(S)$ is a subshift of $X_{[2]}$ such that any $x\in X(S)$ is a concatenation of blocks of the form $0^s1$ with $s\in S$, where points with infinitely many $0$'s to both sides are allowed when $S$ is infinite.
Let $\lambda$ be the unique positive solution to $\sum_{m\in S} x^{-m-1}=1$. Then $h_{top}(X(S))=\log \lambda$ [DJ12], and the mme $\mu_0$ of $X(S)$ is determined by
$$
\frac{\mu_0(X_{0} X_{1} \cdots X_{i+1} = 10^i 1)}{\mu_0(X_{0}=1)}= \lambda^{-i-1} \quad \mbox{ for any $i\in S$}
$$
and
$$
\mu_0 (X_1 \cdots X_n=x_1 \cdots x_n \vert X_{-m} \cdots X_{-1} X_0 =x_{-m} \cdots x_{-1}1) = \mu_0 (X_1 X_2 \cdots  X_n=x_1\cdots x_n \vert X_0=1)
$$
for any $m,n$ and any allowed block $x_{-m}\cdots x_{-1}1x_1\cdots x_n$ [GP19, Corollary 3.9].

It has been proven in [DJ12] that $X(S)$ is an SFT if and only if $S$ is finite or cofinite. Indeed, the forbidden set of $X(S)$ is
\begin{align} \label{standard_forbidden_S_gap}
\mathcal{F}=
\begin{cases}
\{10^m1: m\in \{0,1,2,\cdots, \max{S}\} \setminus S\}\cup \{0^{1+\max{S}}\} \qquad &\mbox{ when $S$ is finite}
\\
\\
\{10^m1: m\in \mathbb{Z}_{\geq 0} \setminus S\} \qquad &\mbox{ when $S$ is cofinite,}
\end{cases}
\end{align}
which will be called the {\em standard forbidden set} of $X(S)$ in this paper.


\section{A Problem in Information Theory}
\label{InfoThy}

A central object in information theory is a discrete channel.  Here, there is a space of input sequences $X$, a space of output sequences $Y$, each over a finite alphabet, and for each $x \in X$ a probability measure $\lambda_x$ on $Y$ which gives the distribution of outputs, given that $x$ was transmitted. One assumes that the map $x \mapsto \lambda_x$ is at least measurable and the channel is stationary in the sense that $\lambda_{\sigma x} = \sigma\lambda_x$ where $\sigma$ is the left shift.

Typically,  $X$ and $Y$ are full shifts and in the simplest case, that of a discrete memoryless channel, $\lambda_x(y_1 \ldots y_n) = \Pi_{i=1}^n p(y_i|x_i)$; here, for each element $a$ of the alphabet of $X$, $p(\cdot|a)$ is a probability distribution on the alphabet of $Y$;
the channel is memoryless in the sense that conditioned on the input $x_i$, the output $y_i$
is independent of all other inputs.
For example, the {\em binary symmetric channel} (BSC)
is the memoryless channel where
$X$ and $Y$ are the full 2-shift and
$$
p(b|a) = \left\{\begin{array}{cc} \epsilon & \qquad b \ne a  \\ 1- \epsilon &\qquad  b = a.
\end{array}\right.
$$
Here, $\epsilon$ is a parameter, known as the crossover probability.

Given a stationary (i.e., shift invariant) input measure $\nu$ on $X$, one defines the stationary output measure $\mu$ on $Y$ by $\mu = \int \lambda_x d\nu$. The {\em mutual information} of $\mu$ and $\nu$ is defined as
$$
I(\mu,\nu) = h(\mu) - h(\mu | \nu) = h(\nu) - h(\nu | \mu)
$$
where $h(\cdot)$ denotes entropy, and $h(\cdot | \cdot)$ denotes conditional entropy
(the second equality follows from the chain rule for entropy, which is a fundamental equality in information theory);
in information theory, shift-invariant measures are viewed as stationary processes,  and these entropies
are often referred to as entropy rates.

There are several notions of channel capacity, which all agree under relatively
mild assumptions.  The {\em stationary capacity} (capacity for short) of a discrete noisy channel is defined
$$
Cap = \sup_{ \mbox{ stationary $\nu$}} I(\mu,\nu).
$$
Note that this makes sense since the output measure $\mu$ is a function of the input measure $\nu$ and the channel.

For a discrete memoryless channel, the capacity can be computed effectively because
it agrees with the $\sup$ when restricted only to i.i.d. (i.e., stationary Bernoulli) measures, turning it into a finite dimensional optimization problem, and, while there
is no known closed form expression for capacity in general, the optimum can be effectively approximated  by the well-known Blahut-Arimoto algorithm [Bla72, Ari72].

We define
the {\em $k$-th order Markov capacity}
$$
Cap_k = \sup_{ \mbox{ stationary $k$-th order Markov $\nu$} } I(\mu,\nu).
$$
We are interested in the problem: {\em when does Markov capacity achieve capacity,
i.e., when does $Cap_k = Cap$ for some $k$}?

It is known, using the ergodic decomposition, that under mild assumptions,
$Cap$ ({\em resp.}, $Cap_k$)
coincides with the maximum mutual information over all
     stationary, {\em ergodic} input measures ({\em resp.}, stationary, {\em irreducible}, $k$-th order Markov input measures) [Gra11, Fei59].

Again, with mild assumptions on the channel, one shows that
$\lim_{k \rightarrow \infty} Cap_k =  Cap$ [CS08]; informally,
``Markov capacity {\em asymptotically} achieves capacity.''
This is important because for fixed $k$, computation of $Cap_k$ is a finite dimensional
optimization problem. According to the discussion above, for discrete memoryless
channels, $Cap_0 = Cap$; informally, ``Bernoulli capacity  achieves capacity.''
But for channels with memory, even just one step of memory, except in certain cases such as input-constrained noiseless channels below, it
is believed that $Cap_k \ne Cap$ for all $k$.  However, little along
these lines has been proven.

If $X$ is not a full shift, then the channel is called {\em input-constrained}.
Typically, the input constraint $X$ is an SFT or sofic shift.  Such a shift space
can be considered a noiseless channel in itself, in a trivial way: $Y  = X$ and for each
$x \in X$, $\nu_x = \delta_x$, the point mass on $\{x\}$.  The capacity of this channel
is easily seen to be the topological entropy, $h_{top}(X)$, otherwise known as
the noiseless capacity, which can be easily computed.

Now, consider the input-constrained binary symmetric channel. This is the BSC, where the inputs are required to belong to a given SFT or
sofic shift $X$ over $\{0,1\}$.  While the capacity of the BSC and the noiseless capacity of $X$ are known explicitly, the capacity of the $X$-constrained BSC
is not known.  And while Markov capacity asymptotically achieves capacity of this channel, it is believed that Markov capacity does not achieve capacity, i.e, for all
$k$, $Cap_k \ne Cap$. However, this has not been proven.

\section{Factor Codes as Channels}
\label{FactorChannel}

This brings us to a main point of our paper: for a class of channels, albeit rather simple in practice, we can rigorously decide whether or not Markov capacity achieves capacity.
An example of this was given in [MPW84, pp. 287-289].
Specifically, we view a factor code $\phi: X \rightarrow Y$ as an {\em input-constrained,
deterministic channel}; here, $\nu_x  = \delta_{\phi(x)}$, so the input determines the
output uniquely. Intuitively, for this channel
input sequences are distorted in a deterministic way.
It follows that, in this case, for any invariant input measure $\nu$, $h(\mu|\nu) =  h(\phi^*(\nu)|\nu) = 0$ where $\phi^*$ is the induced map (of $\phi$) on stationary measures on $X$. So
$$
Cap = \sup_{ \mbox{ stationary $\nu$}} h(\phi^*(\nu)).
$$
According to [MPW84, Corollary 3.2], there exists a stationary input
measure $\nu$ (in fact, a stationary hidden Markov input measure)
such that $\phi^*(\nu) = \mu_0$, the unique mme
on $Y$. Thus, by the variational principle [Wal82, Theorem 8.6], $Cap = h_{top}(Y)$
(an alternative to this argument is to show
that the map $\nu \mapsto \phi^*(\nu)$ is onto the set of all stationary measures on $Y$:
given stationary $\mu$ on $Y$, use the Hahn-Banach theorem to find a not-necessarily-stationary $\nu'$ on $X$ s.t. $\phi^*(\nu') = \mu$  and let $\nu$ be any weak limit point of
the sequence $(1/n)(\nu' +\sigma\nu' + \ldots + \sigma^{n-1}\nu')$).

In summary, we have:

\begin{pr}
\label{Summary}
Let $\phi:X \rightarrow Y$ be a factor code from an irreducible SFT
$X$ to a sofic shift $Y$. Let $\mu_0$ be the unique measure of maximal entropy on $Y$.  For the input-constrained, deterministic channel defined
by $\phi$,
\begin{enumerate}
\item[(1)]
$Cap$ (resp. $Cap_k$) coincides with the {\em maximum} mutual information over all stationary, {\em ergodic} input measures (resp., stationary, {\em irreducible}, $k$-th order Markov input measures);
\item[(2)]
$\lim_{k \rightarrow \infty} Cap_k =  Cap;$
\item[(3)]
$Cap= h_{top}(Y);$
\item[(4)]
A stationary measure $\nu$ on $X$ achieves $Cap$ iff $\phi^*(\nu) = \mu_0$
iff $h(\phi^*(\nu)) = h_{top}(Y).$
\end{enumerate}
\end{pr}


The following simple result gives a relation between P2 and P3.

\begin{pr}~\label{Fto1}
With the same assumptions as in Proposition~\ref{Summary}, if there is an SFT $Z\subset X$ such that $\phi\vert_Z$ is finite-to-one and onto $Y$, then there is an irreducible stationary Markov measure $\nu$ on $Z$ of order at most the memory of $Z$ such that $\phi^*(\nu)=\mu_0$.
\end{pr}

\begin{proof}
Let $\nu$ be the unique mme of any irreducible component of $Z$ with maximum topological entropy. It is stationary, irreducible, and Markov. Since $\phi\vert_Z$ is finite-to-one and onto $Y$,
$$
h(\phi^*(\nu))=h(\nu) =h_{top} (supp~ \nu) =h_{top}(Z) = h_{top}(Y).
$$
Since $\mu_0$ is the unique mme on $Y$, we have $\phi^*(\nu)=\mu_0$.
\end{proof}




\begin{pr}
\label{Equiv_nu}
  Let $\phi:X \rightarrow Y$ be a factor code from an irreducible SFT $X$
to a sofic shift $Y$.
Let $\nu$ be an irreducible stationary Markov measure on $X$ and assume that
$\phi^*(\nu) = \mu_0$,
the unique mme on $Y$ (in particular, Markov capacity achieves capacity
of the
input-constrained deterministic channel determined by $\phi$).

The following are equivalent:

\begin{enumerate}
\item[(1)]
$\phi|_{supp(\nu)}$ is finite-to-one and onto;
\item[(2)]
$h_{top}(supp(\nu))  =  h_{top}(Y)$;
\item[(3)]
$h(\nu) =   h_{top}(Y)$;
\item[(4)]
For every periodic point in $supp(\nu)$ the weight per symbol, for $\nu$, is
$e^{-h_{top}(Y)}$ (the {weight per symbol} of
a periodic point $(p_0 \ldots p_{n-1})^\infty$ for a $k$-th order Markov measure $\nu$ on $X$ is defined to be $\nu(p_0 \ldots p_{n-1} | p_{-k} \ldots p_{-1})^{1/n}$).
\end{enumerate}
\end{pr}

\begin{proof}
{{\em (1)} {$\Rightarrow$} {{\em (2)}} :} This follows directly from Corollary 8.1.2 of [LM95].

{{\em (2)} {{$\Rightarrow$}} {{\em (3)}} :}
$$
 h_{top}(Y)= h_{top}(supp(\nu)) \ge h(\nu) \ge h(\phi^*(\nu)) = h(\mu_0) =  h_{top}(Y).
$$
This yields {\em (3)}.

{{\em (3)} {{$\Rightarrow$}} {{\em (1)}} :}
Apply [Par97, Theorem 2].

{{\em ((2) and (3))} {{$\Rightarrow$}} {{\em (4)}} :} The condition that for some $c \ge 0$,
for every periodic point in $supp(\nu)$ the $\nu$-weight per symbol is $e^{-c}$,
is equivalent to the condition that $h(\nu) = c$
and that $\nu$ is an mme for $supp(\nu)$.
This is essentially contained in [PT82, Proposition 44].
\end{proof}

It follows from Propositions~\ref{Fto1} and~\ref{Equiv_nu} that P2 holds iff P3 holds with a measure $\nu$ that is also irreducible stationary Markov and satisfies any of the equivalent conditions in Proposition~\ref{Equiv_nu}. We will return to this point in Section~\ref{Conjecture}.
\bigskip

\section{Factor Codes with an Unambiguous Symbol}
\label{unambig}

We begin with a brief introduction to factor codes with an unambiguous symbol. Such factor codes are also known as factor codes with a singleton clump [PQS03].


Let $X$ be a shift space over an alphabet $\mathcal{A}$
 and $D=b_1 b_2 \cdots b_k$ be an allowed block in $X$. Define $\Phi: \mathcal{A}^k\rightarrow \{0,1\}$ by
\begin{align}
\Phi(x_{[1,k]})=
\begin{cases} \label{unam_def}
1 \qquad \mbox{ if } x_{[1,k]}=D \\
0 \qquad \mbox{ otherwise}.
\end{cases}
\end{align}
Then, the factor code $\phi: X \rightarrow Y\subset X_{[2]}$ induced by $\Phi$ is called a {\em factor code with an unambiguous symbol}.  Here, $Y$ is the image of $\phi$.

In the remainder of this paper, we focus on the case when $X$ is an irreducible SFT.
Note that in this case, by passing to a higher block shift, in the preceding definition
we can and sometimes will assume that $k=1$ and that $X$ is an SFT with memory 1.


The following propositions give some properties of $Y$.
\begin{pr}
Let $\phi: X\rightarrow Y$ be a factor code with an unambiguous symbol. Then
$Y$ is an $S$-gap shift. 
\end{pr}
\begin{proof}
The elements of $Y$ are arbitrary concatenations of strings of the form $10^s$ with $s\in S$ such that there exists some allowed block $w$ of length $k+s+1$ satisfying the following:
\begin{enumerate}
\item[(1)] $w_{[1,k]}=D$;
\item[(2)] $w_{[s+2,s+k+1]}=D$;
\item[(3)] for all $2\leq i\leq s+1$, $w_{[i,k+i-1]}\neq D$.
\end{enumerate}
Hence, $Y$ is an $S$-gap shift.
\end{proof}

\begin{pr} \label{periodic_and_Y}
Let $\phi: X\rightarrow Y$ be a factor code with an unambiguous symbol. If $X= X_{[2]}$, then
\begin{enumerate}
\item[(1)] $10^{k-1}1$ is not allowed in $Y$ iff $D$ is purely periodic (i.e., $D = u^{\ell}$ for some $\ell \ge 2$
and some block $u$);
\item[(2)] For any $j\geq k$, $10^j1$ is allowed in $Y$.
\end{enumerate}
\end{pr}
\begin{proof}
To prove {\em (1)}, first observe that $10^{k-1}1$ allowed  iff the image of $DD$ is
$10^{k-1}1$.
If $10^{k-1}1$ is not  allowed, then the image of $DD$ has a prefix of the form
$10^c1$ for some $0\le c \le k-2$. Let $d = c+1 \le k-1$.
Then for  all $0\le i \le k-1$, $b_i = b_{i+d}$ (here and below in this proof, subscripts are read modulo $k$).  It follows that for all integers $m,n$ and
all $0\le i \le k-1$, $b_i = b_{i+md+nk}$.  Let $e = \gcd(d,k)$.  Then $e = md+nk$ for
some $m,n$.  Thus, for all $0\le i \le k-1$, $b_i = b_{i+e}$. It follows that $D = b_1 \cdots b_k= (b_1\cdots b_e)^{k/e}$.  Since $e < k$, $k/e \ge 2$. So, $D$ is purely periodic.

Conversely, assume that $D$ is purely periodic.  Then the image of the block $DD$ is not
$10^{k-1}1$ and so $10^{k-1}1$ is not allowed.

We now prove {\em (2)}. For $j\geq k$, we show $10^j1$ is allowed in $Y$ by finding a binary block $x_1x_2\cdots x_{j-k+1}$ such that
\begin{align} \label{suff_cdn_5.2}
\Phi(b_1 \cdots b_k x_1 x_2\cdots x_{j-k+1} b_1 \cdots b_k) =10^j1.
\end{align}

If $b_1 \cdots b_k=0^k$, then one immediately verifies that $\Phi(b_1\cdots b_k 1^{j-k+1} b_1\cdots b_k)=10^j1$. By reversing the roles of $0$ and $1$ in the domain a similar argument works when $b_1 \cdots b_k=1^k$.

Now assume that $b_1\cdots b_k\neq 0^k$ and $b_1\cdots b_k\neq 1^k$. Express $b_1 \cdots b_k$ uniquely by
\begin{align} \label{unique_expre_b_1b_k}
b_1b_2\cdots b_k =(b_1 \cdots b_m)^s b_1 \cdots b_t \qquad (m\geq 2, s\geq 1, 0\leq t<m)
\end{align}
where $ms+t=k$ and $m$ is the smallest positive integer such that $b_1 \cdots b_k$ can be expressed by (\ref{unique_expre_b_1b_k}). We consider the following two cases:

\noindent{\bf \underline{Case 1:} $j-k+1\geq m$.}

In this case, we claim that (\ref{suff_cdn_5.2}) is satisfied by letting $x_1x_2\cdots x_{j-k+1}=1^{j-k+1}$. To see this, assume to the contrary that
$$
\Phi((b_1 \cdots b_m)^s b_1 \cdots b_t 1^{j-k+1} (b_1 \cdots b_m)^s b_1 \cdots b_t) \neq 10^j1.
$$
This means that there is an extra $1$ in addition to the two $1$'s at the first and the last position in the image. Hence, there is an extra $b_1 \cdots b_k$ in the input in addition to the two at the initial and tail end (these two $b_1 \cdots b_k$'s will be called the head and the tail, respectively). Since $x_1 \cdots x_{j-k-1} =1^{j-k+1}$ and $b_1 \cdots b_k\neq 1^k$, this extra $b_1 \cdots b_k$ must start with some $b_1 \cdots b_t$ in the head or end with some $b_1 \cdots b_t$ in the tail. Thus, it must intersect the ``intermediate" subblock $x_1 \cdots x_{j-k+1}$ in at least $m$ bits.
Therefore, either
\begin{align} \label{in_the_head}
x_1 x_2 \cdots x_m= b_{t+1} \cdots b_m b_1 \cdots b_t.
\end{align}
or
\begin{align} \label{in_the_tail}
x_{j-k-m+2}\cdots x_{j-k+1}=b_1 \cdots b_m.
\end{align}
Recalling that $x_1 \cdots x_{j-m+1}=1^{j-k+1}$, either (\ref{in_the_head}) or (\ref{in_the_tail}) implies $b_1 b_2 \cdots b_k=1^k$, a contradiction.

\noindent{\bf \underline{Case 2:} $1\leq j-k+1 < m$.}

In this case, an extra $b_1 \cdots b_k$ in the input must intersect the head, the tail and the ``intermediate" subblock  $x_1 x_2\cdots x_{j-k+1}$ simultaneously. Thus, this extra $b_1 \cdots b_k$ must start with some $b_1 \cdots b_t$ in the head and end with some $b_1 \cdots b_t$ in the tail. Therefore,  (\ref{suff_cdn_5.2}) holds as long as
\begin{align} \label{x_1_and_x_{j-k+1}}
\begin{cases}
x_1 \neq b_{t+1}  \quad \mbox{and}\quad  x_{j-k+1} \neq b_m \qquad &\mbox{if $j-k>0$}\\
x_1 \neq b_{t+1}  &\mbox{if $j-k=0$},
\end{cases}
\end{align}
which is always possible for some binary $x_1x_2 \cdots x_{j-k+1}$.\end{proof}



\section{Characterization of the One-to-One Condition for Factor Codes with an Unambiguous Symbol} \label{P1_for_unam}

In this section, we address P1 for factor codes with an unambiguous symbol. Through this section, a factor code with an unambiguous symbol always refers to the one induced by $\Phi$ in (\ref{unam_def}) unless otherwise specified.

We have the following theorem which characterizes the existence of a subshift of finite type, on which the restriction of $\phi$ is one-to-one and onto.

\begin{thm} \label{one-to-one_SFT}
Let $\phi: X\rightarrow Y$ be a factor code with an unambiguous symbol defined on an irreducible shift space $X$. Let $S$ be such that $Y$ is an $S$-gap shift. Then, there is a shift space $Z\subset X$ s.t. $\phi\vert_Z$ is a conjugacy from $Z$ onto $Y$ if and only if either of the following conditions holds:
\begin{enumerate}
\item[\textnormal{(C1)}] $S$ is a finite set;
\item[\textnormal{(C2)}] there is a fixed point (i.e., fixed via the shift) in $X$ other than $D^\infty$.
\end{enumerate}
Moreover, $Z$ and $Y$ must be SFTs if either (C1) or (C2) holds.
\end{thm}
\noindent(Note: $D^{\infty}$ may or may not be in $X$ and even if $D^\infty\in X$, it may or may not be a fixed point.)
\begin{rem}
Note that to say that $S$ is finite means that there exists some $M$ such that
every allowed block in $X$ of length $M$ contains $D$ as a subblock. Sometimes, one says that in such a case $D$ is a ``Rome".
\end{rem}

\begin{rem}
According to Proposition \ref{Fto1}, when (C1) or (C2) holds, the capacity of the deterministic channel, defined by $\phi$, is achieved by a Markov chain.
\end{rem}

\noindent{\em Proof of Theorem \ref{one-to-one_SFT}:}
We first show that $Y$ must be an SFT (and thus $Z$ must also be an SFT) when (C1) or (C2) holds. This follows from the
fact that an $S$-gap shift is an SFT iff $S$ is either finite or cofinite [DJ12]. If (C1) holds, there is nothing to prove. If (C2) holds, then, using irreducibility of $X$, the image $Y$ contains points with blocks that begin and end with $1$ and contain arbitrarily long strings of $0$'s in between, and thus $S$ is cofinite.


{\bf Only if part:} If $S$ is finite, we are done. So assume that $S$ is infinite. Then $0^\infty\in Y$. Since there exists a shift space $Z\subset X$ s.t. $\phi\vert_Z$ is a conjugacy from $Z$ onto $Y$, $Z$ must have a fixed point $z$ s.t. $\phi(z)=0^\infty$. Finally, noting that $D^\infty \notin X$ or $\phi(D^\infty) \neq 0^\infty$, we conclude that $z$ must be different from $D^\infty$.

{\bf If part:} Assume Condition (C2) of the theorem. Up to recoding, we may assume that $X$ is a (1-step) vertex shift $\widehat{X_G}$, $D$ is a vertex of the graph $G$ and there is a vertex $A$ in $G$ such that $A$ is distinct from $D$ and $G$ has a self-loop  $\tau$ at $A$. Using irreducibility of $X$, there are paths in $G$, $\beta^+$ from $D$ to $A$ and $\beta^-$ from $A$ to $D$, neither of which contains $D$ in its interior. Let $N: =\vert \beta^+\beta^-\vert-1$.

Now $Y$ is a gap shift with gap set of the form $S:= F \cup \{N, N+1, \cdots\}$, where each element of $F$ is less than $N$. For each $s\in S$, choose $\pi^s$ to be a first-return cycle of length $s$ from $D$ to itself (``first-return" means that it does not contain $D$ in its interior). We will assume that for $s\geq N$, we  choose $\pi^s=\beta^+ \tau^{s-N} \beta^-$. For $y\in Y$, let $O_y:= \{j\in \mathbb{Z}: y_j=1\}$ and define $\eta: Y \rightarrow X$ as follows:
\begin{enumerate}
\item[(D1)] if $i\in O_y$, define $(\eta(y))_i=D$;
\item[(D2)] if $j,j'\in O_y$ and $\{l\in \mathbb{Z}: j<l<j'\}\subset O_y^c$, define $(\eta(y))_{[j,j']}=V(\pi^{j'-j})$;
\item[(D3)] If $O_y$ has a maximum element $s$, define $(\eta(y))_{[s,\infty)}=V(\beta^+\tau^{\infty})$;
\item[(D4)] if $O_y$ has a minimum element $s$, define $(\eta(y))_{(-\infty, s]}=V(\tau^{\infty} \beta^-)$;
\item[(D5)] if $O_y=\emptyset$, define $\eta(y)=A^\infty$.
\end{enumerate}

Observe that $\eta$ is injective because if $y,y'\in Y$ and $y\neq y'$, then for some $i$, WLOG we assume $y_i=1$ and $y_i'=0$ and so $(\eta(y))_i=D$ and $(\eta(y'))_i\neq D.$ Furthermore, we claim that $\eta$ is a sliding block code. To see this, note that $\eta$ is shift-invariant by virtue of its definition, and $(\eta(y))_i$ is a function of $y_{[-N+i, N+i]}$.

So, $\eta$ is an injective sliding block code from $Y$ into $X=\widehat{X_G}$. Let $Z$ be its image. Then, $\eta^{-1}$ is a bijective sliding block code from $Z$ onto $Y$.
Moreover, by the construction of $\eta$, for every $y\in Y$,
\begin{equation} \label{phi-eta-inverse}
\phi \circ \eta(y) =y.
\end{equation}
It follows that $\eta^{-1} = \phi \vert_Z$. This completes the proof of the if part assuming Condition (C2).

Now assume Condition (C1). The proof follows along the same lines except that the definition of $\eta$ is even easier: $S=F$ is a finite set, and we only need the first two cases, (D1) and (D2), of the definition of $\eta$ because for any $y\in Y$, $O_y$ is a nonempty set with no maximum and no minimum. \qed

\begin{exmp} \label{1010->1}
Let $\mathcal{F}_1=\{111\}$, $X=X_{\mathcal{F}_1}$ and $\Phi: \{0,1\}^4 \rightarrow \{0,1\}$ be a $4$-block code defined by
$$
\Phi(x_{[1,4]})=
\begin{cases}
1 \quad \mbox{if } x_{[1,4]}=1010\\
0 \quad \mbox{otherwise}.
\end{cases}
$$
We let $\phi: X\rightarrow Y$ be the factor code with an unambiguous symbol induced by $\Phi$. Using Proposition \ref{periodic_and_Y}, one can verify that $Y$ is an $S$-gap shift with $S=\{1,4,5,6, 7\cdots\}$. Equivalently, $Y$ is an SFT with the forbidden set $\mathcal{F}=\{11,1001,10001\}$. Moreover, since $0^\infty\in X$, Condition (C2) is satisfied and we conclude from Theorem \ref{one-to-one_SFT} that there is an SFT $Z\subset X$ such that $\phi\vert_Z$ is a conjugacy from $Z$ to $Y$.
\end{exmp}


When the domain of $\phi$ is $X_{[2]}$, then Condition (C2) in Theorem \ref{one-to-one_SFT} holds and there is always an SFT $Z \subset X$ to which the restriction of $\phi$ is one-to-one and onto $Y$. Note that $Y$ must be an $S$-gap shift with $S$ cofinite. Our next result gives an explicit description of $Z$ for some special cases.


\begin{thm} \label{domain_full_one-to-one}
Let $\phi:X=X_{[2]} \rightarrow Y$ be a factor code with an unambiguous symbol,
$\mathcal{F}$ be the standard forbidden set of $Y$ and $\overline{\mathcal{F}}$ be the bitwise complement of $\mathcal{F}$. Then, the following are equivalent:
\begin{enumerate}
\item[(1)] At least one of the symbols from $\{0,1\}$ occurs at most once in $D$;
\item[(2)] Either $\phi\vert_{X_\mathcal{F}}$ or $\phi\vert_{X_{\overline{\mathcal{F}}}}$ is one-to-one and onto $Y$;
\item[(3)] Either $\phi\vert_{X_\mathcal{F}}$ or $\phi\vert_{X_{\overline{\mathcal{F}}}}$ is finite-to-one and onto $Y$;
\item[(4)] Either $\phi\vert_{X_{\mathcal{F}}}$ or $\phi\vert_{X_{\overline{\mathcal{F}}}}$ is onto $Y$.
\end{enumerate}

\noindent(Note: When {(1)} holds, $\phi\vert_{X_\mathcal{F}}$ and $\phi\vert_{X_{\overline{\mathcal{F}}}}$ may not both satisfy {(2)} ({resp.}, {(3) and (4)}). For example, suppose $k=4$ and $D=b_1 b_2 b_3 b_4=0000$. Then, one verifies that $\phi\vert_{X_\mathcal{F}}$ is one-to-one and onto, but $\phi\vert_{X_{\overline{\mathcal{F}}}}$ is not. See Example \ref{0000_maps_to_1} for more details.)
\end{thm}

\begin{proof}
When $k=1$, $Y=X=X_{[2]}$ and $\phi$ is trivially a conjugacy. Hence, we assume $k\ge2$ throughout the remainder of the proof.

{{\em (1)} {$\Rightarrow$} {{\em (2)}} :} We consider the following two cases.\\
\noindent{\bf \underline{Case 1}: $b_1 \cdots b_k=0^{k}$ or $b_1 \cdots b_k=1^k$.}

Assume $b_1 \cdots b_k=0^{k}$.
Then, $Y$ is an $S$-gap shift with $S=\{0, k, k+1, \cdots\}$. Equivalently, $Y$ is an SFT with forbidden set
$$\mathcal{F}=\{101,1001, \cdots, 10^{k-1}1\}.$$

Note that any $y\in Y$ can be uniquely expressed by $y=\cdots 1^{m_1}0^{n_1} 1^{m_2}0^{n_2} 1^{m_3} \cdots$ with $m_i\geq 1, n_i\geq k$. Define
$$x:= \cdots 0^{m_1+k-1}1^{n_1-k+1} 0^{m_2+k-1} 1^{n_2-k+1} 0^{m_3+k-1} \cdots.$$ Then, $x\in X_{\mathcal{F}}$ and $\phi(x)=y$. Hence, $\phi\vert_{X_{\mathcal{F}}}$ is onto.

We then claim that $\phi \vert_{X_\mathcal{F}}$ is one-to-one. To see this, consider $x, x'\in X_{\mathcal{F}}$ and $x\neq x'$. Then, for some $i$, WLOG we assume $x_i=1, x_i'=0.$ Now, $x_i=1$ implies $(\phi(x))_{[i, i+k-1]}=0^k$; on the other hand, recalling that $\mathcal{F}=\{101,1001, \cdots, 1{0}^{k-1}1\}$, we deduce from $x'_i=0$ that there is an $i\leq l\leq i+k-1$ such that $x'_{[l-k+1, l]}=0^k$ and therefore $(\phi(x'))_l=1$. Thus, $\phi(x)\neq \phi(x')$ and $\phi\vert_{X_{\mathcal{F}}}$ is one-to-one.





By reversing the roles of $0$ and $1$ in the domain it follows that $\phi \vert_{X_{\overline{\mathcal{F}}}}: X_{\overline{\mathcal{F}}} \rightarrow X_{\mathcal{F}}$ is also one-to-one and onto when $b_1 \cdots b_k={1}^{k}$.
\medskip

\noindent{\bf \underline{Case 2}: There is only one $0$ or only one $1$ in $b_1 b_2 \cdots b_k$.}

We first assume that
$b_j=1$ for some $1\leq j\leq k$ and $b_i=0$ for any $1\leq i\leq k$ and $i\neq j$. 
Let $M:=\max\{j-1, k-j\}$. Then $Y$ is an $S$-gap shift with $S=\{M, M+1, \cdots\}$. Equivalently, $Y$ is an SFT with the forbidden set $\mathcal{F}=\{11, 101, \cdots , 10^{M-1} 1\}$. By expressing any $x\in X_{\mathcal{F}}$ by
$$
x=\cdots 1 0^{m_{-1}} 1 0^{m_0} 1 0^{m_1} 1 \cdots
$$
with $m_l\geq M$ for all $l\in \mathbb{Z}$, one directly verifies that $\phi(x) = \sigma^{j-k} (x)$.
Thus, $\phi\vert_{X_\mathcal{F}}$ must be one-to-one and onto $Y$.

By reversing the roles of $0$ and $1$ in the domain it follows that $\phi\vert_{X_{\overline{\mathcal{F}}}}\rightarrow X_{\mathcal{F}}$ is also one-to-one and onto when there is only one $0$ in $b_{[1,k]}$.
\medskip

{ {\em (2)} $\Rightarrow$ {\em (3)}:} Obvious.
\medskip

{{\em (3)} $\Rightarrow$ {\em (4)}:} Obvious.
\medskip

{{\em (4)} $\Rightarrow$ {\em (1)}:}
We prove by way of contradiction. Suppose there are at least two $1$'s and at least two $0$'s in $b_{[1,k]}$. Then, $k\geq 4$ and
$11\in \mathcal{F}$.
We will show that both $\phi\vert_{X_{\mathcal{F}}}$ and $\phi\vert_{X_{\overline{\mathcal{F}}}}$ are not onto by finding a $y\in Y$ and two blocks $B_1\in \mathcal{F}$ and $B_2\in \overline{\mathcal{F}}$ such that any $x\in \phi^{-1}(y)$ contains $B_1$ and $B_2$. Indeed, if such a $y$ exists, then $y\notin \phi(X_{\mathcal{F}})$ and $y\notin \phi(X_{\overline{\mathcal{F}}})$ and therefore both
$\phi\vert_{X_{\mathcal{F}}}$ and $\phi\vert_{X_{\overline{\mathcal{F}}}}$
are not onto, contradicting {\em (4)}.

We consider the following cases:

\noindent{\bf \underline{Case 1:}  Both $00$ and $11$ are subblocks of $b_1 b_2 \cdots b_k$}.

Choose $y\in Y$ with $y_0=1$. Then, for any $x\in \phi^{-1}(y)$, $x_{[-k+1,0]}=b_{[1,k]}$. Since $11\in \mathcal{F}$, $00\in \overline{\mathcal{F}}$ and they are both subblocks of $x$, we conclude that $\phi\vert_{X_{\mathcal{F}}}$ and $\phi\vert_{X_{\overline{\mathcal{F}}}}$ are not onto.
\medskip

\noindent{\bf \underline{Case 2:} Neither $00$ nor $11$ is a subblock of $b_1 b_2 \cdots b_k$}.

In this case, $b_1 b_2 \cdots b_k$ is a binary block with $0$ and $1$ occurring alternately.
We assume WLOG that $b_1b_2\cdots b_k=010101\cdots$.

If $k$ is odd,
one verifies that $b_1=b_k=0$, $\mathcal{F}=\{10^j1: j\in \{0, 2, 3, \cdots k-2\}\}$ and $\overline{\mathcal{F}}=\{01^j0:  j\in \{0, 2, 3, \cdots, k-2\}\}$.
Consider $y\in Y$ such that $y_{[0, k]}=1 0^{k-1} 1$. For any $x\in \phi^{-1}(y)$, $x_{[-k+1, k]}=(b_1 b_2 \cdots b_k)^2$; in particular, $x_{[-1, 2]}=b_{k-1}b_kb_1b_2=1001\in \mathcal{F}$ and $x_{[0,1]}=b_kb_1=00 \in \overline{\mathcal{F}}$. Thus, both $\phi\vert_{X_{\mathcal{F}}}$ and $\phi\vert_{X_{\overline{\mathcal{F}}}}$ are not onto.

 If $k$ is even,
 $\mathcal{F}=\{10^j1:  j\in \{0, 2, 3, \cdots, k-1\}\}$ and $\overline{\mathcal{F}}=\{01^j0:  j\in \{0, 2, 3, \cdots, k-1\}\}$.
 Consider $y\in Y$ such that $y_{[0,k+1]}=10^k1$. Then for any $x\in \phi^{-1}(y)$, either $x_{[-k+1, k+1]}=b_1b_2\cdots b_k 0 b_1 b_2\cdots b_k$ or $x_{[-k+1, k+1]}=b_1b_2\cdots b_k 1 b_1 b_2\cdots b_k$. In the former case, $x_{[0,3]}=1001\in \mathcal{F}$ and $x_{[0,1]}=00\in\overline{\mathcal{F}}$; in the latter case, $x_{[0,1]}=11\in \mathcal{F}$ and $x_{[-1,2]}=0110\in \overline{\mathcal{F}}$. Therefore $\phi\vert_{X_{\mathcal{F}}}$ and $\phi\vert_{X_{\overline{\mathcal{F}}}}$ are not onto in both cases.

 \noindent{\bf \underline{Case 3:} Exactly one of $00$ or $11$ is a subblock of $b_1b_2\cdots b_k$.}

We assume WLOG that $11$ is a subblock of $b_1b_2\cdots b_k$ yet $00$ is not. If for any $2\leq j\leq k-2$, $01^j0$ is not a subblock of $b_1b_2 \cdots b_k$, then
$b_1b_2 \cdots b_k=1^{m_1} (01)^{m_2} 1^{m_3}$ where either $m_1\geq 2, m_2 \geq 2, m_3 \geq 0$ or $m_1\geq 0, m_2 \geq 2, m_3 \geq 1$.
In either case, one directly verifies that  $11\in \mathcal{F}$, $010\in \overline{\mathcal{F}}$. Consider any $y\in Y$ with $y_0=1$. Then, any $x\in \phi^{-1} (y)$ satisfies $x_{[-k+1, 0]}= b_1 b_2 \cdots b_k$ and therefore it contains both $11$ and $010$. Thus, both $\phi\vert_{X_{\mathcal{F}}}$ and $\phi\vert_{X_{\overline{\mathcal{F}}}}$ are not onto.

Otherwise, there exists  $2\leq j\leq k-2$ such that $01^{j}0$ is a subblock of $b_1 b_2 \cdots b_k$.
If $b_1 b_2\cdots b_{k-j-1}\neq b_{j+2}\cdots b_k$, then $10^j1 \in \mathcal{F}$ and therefore $01^j0 \in \overline{\mathcal{F}}$. Let $y\in Y$ be such that $y_0=1$. Then, for any $x\in\phi^{-1}(y)$, $x_{[-k+1, 0]}=b_1 b_2 \cdots b_k$ and therefore $x$ contains both  $11\in \mathcal{F}$ and $01^j0 \in \overline{\mathcal{F}}$. Hence,
both $\phi\vert_{X_{\mathcal{F}}}$ and $\phi\vert_{X_{\overline{\mathcal{F}}}}$ are not onto.

If $b_1 b_2\cdots b_{k-j-1}= b_{j+2}\cdots b_k$, then
{\small
$$
b_1 b_2 \cdots b_k=
\begin{cases}
1^{s_1} (01^j)^{m_1} \qquad \ \ \mbox{ with } 0\leq s_1\leq j, m_1\geq 2 \mbox{ and } s_1+m_1(j+1)=k
\medskip
\\
\mbox{ or }
\medskip
\\
1^{s_2} (01^j)^{m_2} 01^{t_2} \quad \mbox{ with } 0\leq s_2 \leq j, m_2\geq 1, 0\leq t_2\leq j-1 \mbox{ and } s_2+m_2(j+1)+t_2+1=k,
\end{cases}
$$
}
and $10^{i}1\in \mathcal{F}$ for any $j+1\leq i\leq 2j$.

\begin{itemize}
\item[]{\bf \underline{Subcase 3.1:} $b_1 b_2 \cdots b_k=1^{s_1} (01^j)^{m_1}$ for some $0\leq s_1\leq j$ and $m_1\geq 2$.}

If $s_1=0$, $b_1b_2\cdots b_k=(01^j)^{m_1}$ and it is purely periodic.
In this case, we infer from Proposition~\ref{periodic_and_Y} {\em (1)} that $10^{k-1}1$ is not allowed in $Y$ but $10^{k}1$ is. Consider $y\in Y$ with $y_{[0,k+1]}=10^k1$. For any $x\in \phi^{-1}(y)$, either $x_{[-k+1, k+1]}=b_1 b_2\cdots b_k 0 b_1 b_2\cdots b_k = (01^j)^{m_1} 0 (01^j)^{m_1}$ or $x_{[-k+1, k+1]}=b_1 b_2\cdots b_k 1 b_1 b_2\cdots b_k = (01^j)^{m_1} 1 (01^j)^{m_1}$. In the former case, $x_{[0,1]}=00\in \overline{\mathcal{F}}$; in the latter case, $x_{[-j-1,1]}=01^{j+1}0 \in \overline{\mathcal{F}}$. Since $b_1 b_2 \cdots b_k$ contains $11\in \mathcal{F}$, we conclude that both $\phi\vert_{X_{\mathcal{F}}}$ and $\phi\vert_{X_{\overline{\mathcal{F}}}}$ are not onto.

If $s_1\neq 0$, $b_1b_2\cdots b_k$ is not purely periodic. Hence, we infer from Proposition \ref{periodic_and_Y} {\em (1)} that $10^{k-1}1$ is allowed in $Y$. A similar argument as in Case 2 for odd $k$ implies that both $\phi\vert_{X_{\mathcal{F}}}$ and $\phi\vert_{X_{\overline{\mathcal{F}}}}$ are not onto.

 \item[] {\bf \underline{Subcase 3.2:} $b_1 b_2 \cdots b_k=1^{s_2} (01^j)^{m_2}01^{t_2}$ for some $0\leq s_2\leq j$, $m_2\geq 1$ and $0\leq t_2 \leq j-1$.}

If $s_2=j$ and $t_2=0$, $b_1b_2\cdots b_k=(1^j0)^{m_2}$. By reversing the roles of $0$ and $1$, a similar argument as in Subcase 3.1 for $s_1=0$ implies that both $\phi\vert_{X_{\mathcal{F}}}$ and $\phi\vert_{X_{\overline{\mathcal{F}}}}$ are not onto.

If $s_2\neq j$ or $t_2\neq0$, a similar argument as in Subcase 3.1 for $s_1 \neq 0$ again implies that both $\phi\vert_{X_{\mathcal{F}}}$ and $\phi\vert_{X_{\overline{\mathcal{F}}}}$ are not onto.
\end{itemize}
\end{proof}

\begin{exmp} \label{0000_maps_to_1}
Let $\Phi: \{0,1\}^2 \rightarrow \{0,1\}$ be a $4$-block code defined by
$$
\Phi(0000)=1 \quad \mbox{and} \quad \Phi(b_1 b_2 b_3 b_4)=0 \quad \mbox{ if } \quad b_1 b_2 b_3 b_4\neq 0000.
$$
Let $\phi: X=X_{[2]}\rightarrow Y$ be the factor code induced by $\Phi$. Using Proposition \ref{periodic_and_Y},  one verifies that $Y$ is an $S$-gap shift with $S=\{0, 4,5,6, \cdots\}$. Equivalently, $Y$ is an SFT with the forbidden set $\mathcal{F}=\{101, 1001, 10001\}$. Noting that $1^\infty\in X$, we deduce from Theorem \ref{one-to-one_SFT} that there is an SFT $Z\subset X$ such that $\phi\vert_Z$ is a conjugacy. Note that $\phi\vert_{X_{\overline{\mathcal{F}}}}$ is not onto: since $010\in \overline{\mathcal{F}}$ and $\Phi^{-1}(100001)=000010000$, $100001$ is not allowed in the image of $\phi\vert_{X_{\overline{\mathcal{F}}}}$ and therefore $\phi\vert_{X_{\overline{\mathcal{F}}}}$ is not onto.
It follows from Theorem~\ref{domain_full_one-to-one} that we can choose $Z$ to be $X_{\mathcal{F}}$. The reader can verify this directly.
\end{exmp}

\begin{rem} \label{f-to-1_but_not_1-to-1}
If a factor code $\phi$ defined on an irreducible SFT $X$ is finite-to-one but not one-to-one, then there is no shift space $Z\subset X$ such that $\phi\vert_Z$ is one-to-one and onto. This follows from the fact that if such a $Z$ exists, then by [LM95, Corrolary 4.4.9], $h_{top}(Z)<h_{top}(X)$, which contradicts [LM95, Corollary 8.1.20]. For a simple example of such a $\phi$ with an unambiguous symbol, see Example \ref{conti_example}.
\end{rem}


\section{Standard Factor Codes Defined on Spoke Graphs} \label{Sec_Spoke}
In this section, we consider a class of factor codes with an unambiguous symbol motivated by the example in
[MPW84, pp. 287-289].




A graph $U$ is called a {\em spoke} if $U$ consists of a state
$B$, a simple path $\gamma^+$ from $B$ to a state $B'\ne B$,
 a simple path $\gamma^-$ from $B'$ to $B$, a simple cycle $C$ including $B'$ s.t.
 $\gamma^+, \gamma^-$ and $C$ are all disjoint (except that they all share the state
 $B'$ and  $\gamma^+, \gamma^-$ share the state $B$). We also allow degenerate spokes with one simple cycle $C$ at $B$, which we indicate by $\gamma^+=\gamma^-=\emptyset.$


A graph $G$ is a {\em spoke graph} if it consists of a central state
 $B$ and finitely many distinct spokes $U_i, i \in T$ such that for any $i\neq j\in T$, $U_i$ and $U_j$ only intersect at $B$. Let $\gamma^+_i, \gamma^-_i,B_i'$ and $C_i$ denote the
  $\gamma^+, \gamma^-, B'$ and $C$ of the spoke $U_i$.  Let $T_0\triangleq \{i\in T:  \gamma^+=\gamma^-=\emptyset\}$ denote the indices of degenerate spokes and $T_1\triangleq T\setminus T_0$ denote the indices of regular spokes. See Figure \ref{spoke_graph1} for an example of a spoke graph with two regular spokes and one degenerate spoke.

\begin{figure}[]
\begin{center}
\begin{tikzpicture}[scale=0.6]
\draw [opacity=0] (-3,2) grid (22,12);
\node [circle, draw, thick] (a) at (12,6) {$B$};
\node [circle, fill, inner sep=1.2pt, opacity=0] (a1) at (11.6,5.51) {};
\node [circle, fill, inner sep=1.2pt, opacity=0] (a2) at (11.3,6.0) {};
\node [circle, fill, inner sep=1.2pt, opacity=0] (a3) at (11.3,6.0) {};
\node [circle, fill, inner sep=1.2pt, opacity=0] (a4) at (11.6,6.51) {};
\node [circle, draw, thick] (b) at (5.5,4) {$B_1'$};
\node [circle, draw, thick, opacity=0] (d1) at (9.8, 4.89) {};
\node [circle, fill, inner sep=1.8pt] at (9.8, 4.89) {};
\node [circle, draw, thick, opacity=0] (d2) at (8.0, 4.27) {};
\node [circle, fill, inner sep=1.8pt] at (8.0, 4.27) {};
\node [circle, draw, thick, opacity=0] (d3) at (8.75,5.15) {};
\node [circle, fill, inner sep=1.8pt] at (8.75,5.15) {};
\node [circle, draw, thick, opacity=0] (d4) at (8.75,6.825) {};
\node [circle, fill, inner sep=1.8pt] at (8.75,6.825) {};
\node [circle, fill, inner sep=1.2pt, opacity=0] (b1) at (6.2,3.65) {};
\node [circle, fill, inner sep=1.2pt, opacity=0] (b2) at (6.2,4.30) {};
\node [circle, draw, thick] (c) at (5.5,8) {$B_2'$};
\node [circle, fill, inner sep=1.2pt, opacity=0] (c1) at (6.2,7.65) {};
\node [circle, fill, inner sep=1.2pt, opacity=0] (c2) at (6.2,8.30) {};
\draw [-stealth, black, thick] (a1) to (d1);
\draw [-stealth, black, thick] (d1) to (d2);
\draw [-stealth, black, thick] (d2) to (b1);

\draw [-stealth, black, thick] (b2) to (d3);
\draw [-stealth, black, thick] (d3) to (a2);

\draw [-stealth, black, thick] (a3) to (d4);
\draw [-stealth, black, thick] (d4) to (c1);

\draw [-stealth, black, thick] (c2) to (a4);

\node [circle, draw, thick, opacity=0] (e1) at (14.5, 8.5) {};
\node [circle, fill, inner sep=1.8pt] at (14.5, 8.5) {};
\node [circle, draw, thick, opacity=0] (e2) at (17, 6) {};
\node [circle, fill, inner sep=1.8pt] at (17, 6) {};
\node [circle, draw, thick, opacity=0] (e3) at (14.5, 3.5) {};
\node [circle, fill, inner sep=1.8pt] at (14.5, 3.5) {};
\draw [-stealth, black, thick] (a) to (e1);
\draw [-stealth, black, thick] (e1) to (e2);
\draw [-stealth, black, thick] (e2) to (e3);
\draw [-stealth, black, thick] (e3) to (a);

\node [circle, fill, inner sep=1.8pt] at (0.75, 2.45) {};
\node [circle, fill, inner sep=1.5pt, opacity=0] (f11) at (1.05, 2.2) {};
\node [circle, fill, inner sep=1.5pt, opacity=0] (f12) at (1.05, 2.7) {};
\node [circle, fill, inner sep=1.5pt, opacity=0] (b3) at (4.8, 3.3) {};
\node [circle, fill, inner sep=1.5pt, opacity=0] (b4) at (4.7, 3.8) {};
\draw [-stealth, black, thick] (b3) to (f11);
\draw [-stealth, black, thick] (f12) to (b4);

\node [circle, draw, thick, opacity=0] (g1) at (1, 7) {};
\node [circle, fill, inner sep=1.8pt] at (1, 7) {};
\node [circle, draw, thick, opacity=0] (g2) at (2.2, 11.2) {};
\node [circle, fill, inner sep=1.8pt] at (2.2, 11.2) {};
\draw [-stealth, black, thick] (c) to (g1);
\draw [-stealth, black, thick] (g1) to (g2);
\draw [-stealth, black, thick] (g2) to (c);
\end{tikzpicture}
\end{center}
\caption{A spoke graph with two regular spokes and one degenerate spoke, where dots denote vertices.}
\label{spoke_graph1}
\end{figure}
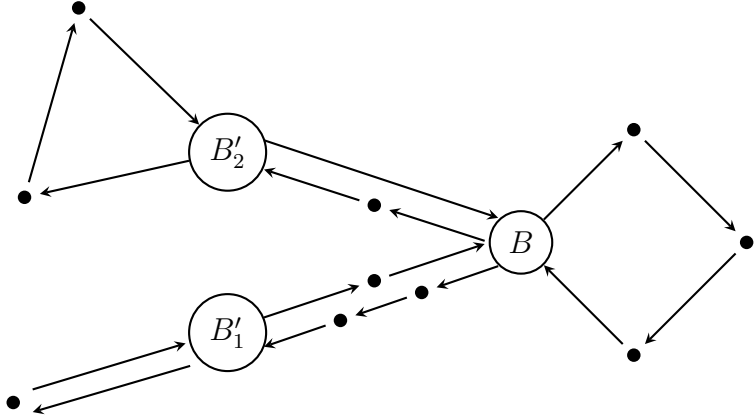

Let $\Phi: \mathcal{V}(G) \rightarrow \{0,1\}$ be defined by
\begin{align} \label{unam_map}
\Phi(x)=
\begin{cases}
1 \quad \mbox{if $x_i=B$} \\
0 \quad \mbox{otherwise}.
\end{cases}
\end{align}
For a block $x_1 \cdots x_m$ with $x_i\in \mathcal{V}(G)$ for any $1\leq i\leq m$, we use $\Phi(x_1 \cdots x_m)$ to denote $\Phi(x_1)\Phi(x_2) \cdots \Phi(x_m)$.

Consider the factor code $\phi: \widehat{X_G} \rightarrow Y\subset X_{[2]}$ induced by $\Phi$. We call $\phi$ the {\em standard factor code} on $G$.
The image $Y$ of $\phi$ is a
 gap shift with gap set
 $$
 S:= \cup_{i\in T} S_i 
 $$
 where
 \begin{align*}
 S_i&:=
 \begin{cases}
  \{d_i -1 \}\qquad \qquad \qquad \qquad \qquad \qquad \quad \  \mbox{if $i\in T_0$}
 \medskip
 \\
 \{n \in \mathbb{Z}_{\geq 0}: n=a_i  \ (\bmod \ d_i), n \ge m_i\} \quad  \mbox{if $i\in T_1$},
 \end{cases} \\
  d_i &:= |C_i| \quad \qquad \qquad \quad  \ \ \ i\in T_0\cup T_1, \\
 m_i &:= |\gamma^+_i| + |\gamma^-_i|-1 \quad \ \ \  i\in T_1, \\
   a_i&: = m_i \bmod d_i \qquad \qquad   0\leq a_i\leq d_i-1.
 \end{align*}
Let $D = l.c.m.(\{d_i: i\in T_1\})$ and $n(i):= D/{d_i}$. It is then immediate that for $i\in T_1$,
$$
S_i = \{n\in \mathbb{Z}_{\geq 0}: n=b_i^{(j)} \  (\bmod \  D), 1\leq j\leq n(i), n\geq m_i \},
$$
where $b_i^{(j)}:=a_i+(j-1) d_i$ and $0\leq b_i^{(j)}<D$ for any $i\in T_1$ and any $1\leq j\leq n(i)$. 
For each $i\in T_1$, denote
$$
K_i := \{b_i^{(1)}, b_i^{(2)}, \cdots, b_i^{n(i)}\} \quad \mbox{and} \quad K_i \bmod D:= \cup_{j=1}^{n(i)}\{n: n=b_i^{(j)}\ (\bmod \ D)\}.
$$
Then the gap set $S$ can be expressed by
$$
S= (\cup_{i\in T_1} \{n \in \mathbb{Z}_{\geq 0}:  n\in K_i \bmod D, n\geq m_i \}) \cup \{\vert C_i \vert -1: i\in T_0\}.
$$

\section{Characterization of the Finite-to-one Condition for Standard Factor Codes on Spoke Graphs} \label{P2_for_spoke}

Here, we characterize P2 for standard factor codes on spoke graphs.

\begin{thm} \label{spoke_three_eq}
Let $G$ be a spoke graph and $\phi$ be the standard factor code on $G$.
Then, the following are equivalent:
\begin{enumerate}
\item[(1)] There is a $W\subset T_1$ such that
$\cup_{i\in W} K_i=\cup_{i\in T_1} K_i$ and $\{K_i: i\in W\}$ are pairwise disjoint.
\item[(2)] There is an irreducible SFT $Z\subset \widehat{X_G}$ such that $\phi\vert_{Z}$ is almost invertible and onto  $Y$.
\item[(3)] There is an irreducible SFT $Z\subset \widehat{X_G}$ such that $\phi\vert_{Z}$ is finite-to-one and onto $Y$.
\end{enumerate}
\noindent({Note 1:}  If $d_i \ge 2$ for all $i \in T_0\cup T_1$, then the vertex shift of a spoke graph does not have a fixed point. If $T_1\neq \emptyset$, then the image $Y$ always has a fixed point $0^\infty$.  So, under these assumptions, $\phi\vert_Z$ cannot be a conjugacy.\\
Note 2: In (2) and (3), it is not necessary to assume that $Z$ is irreducible since otherwise we can replace $Z$ with an irreducible component with maximal topological entropy.)
\end{thm}

\begin{proof}
{\em (1)} $\Rightarrow$ {\em (2)}:
Suppose there is a set $W\subseteq T_1$ s.t. $\cup_{i \in W} K_i  = \cup_{i \in T_1} K_i$
and $\{K_i: i \in W\}$ are pairwise disjoint.

Denote
\begin{align*}
S^{(0)}&=\cup_{i\in W} \{n \in \mathbb{Z}_{\geq 0}:  n\in K_i \bmod D, n\geq m_i \}, \\
S^{(1)}&=\cup_{i\in T_1} \{n \in \mathbb{Z}_{\geq 0}:  n\in K_i \bmod D, n\geq m_i \},\\
S^{(2)}&=\{\vert C_i \vert -1: i\in T_0\}.
\end{align*}
We first construct a new graph $H$ from the graph $G$ through the following three steps:
\begin{enumerate}
\item[$(A)$] Let $H$ be the graph consisting of the central state $B\in \mathcal{V}(G)$ and all the spokes $U_i\subset G$ with $i\in W$;
\item[$(B)$] For each $r\in S^{(1)} \setminus S^{(0)}$, add to $H$ a simple cycle, denoted $C(r)$, of length $r+1$ starting and ending with $B$;
\item[$(C)$] For each $s\in S^{(2)} \setminus S^{(1)}$, choose an $i(s)\in T_0$ such that $\vert C_i \vert = s+1$. Add the degenerate spoke $U_{i(s)}$ to $H$.
\end{enumerate}
See Figure \ref{cons_graph_H} for an example of the construction of $H$.

Let $H_1, H_2, H_3$ denote the subgraph consisting of spokes added to $H$ in Step $(A)$, $(B)$ and $(C)$, respectively. It is worth noting that any $r\in S^{(1)}\setminus S^{(0)}$ corresponds to a ``gap" in regular spokes of $G$ that is missing from $\{U_i: i\in W\}$, and any $s\in S^{(2)}\setminus S^{(1)}$ corresponds to a ``gap" in degenerate spokes of $G$ that is missing from $\{U_i: i\in T_1\}$.

\begin{figure}[]
\begin{center}
\begin{tikzpicture}[scale=0.6]
\draw [opacity=0] (0,-2) grid (25,14);
\node [circle, draw, thick] at (6,5) {$B$};
\node [circle, fill, inner sep=1.8pt, opacity=0] (B1) at (5.6,5.4) {};
\node [circle, fill, inner sep=1.8pt, opacity=0] (B2) at (6.4,5.4) {};
\node [circle, fill, inner sep=1.8pt, opacity=0] (B3) at (5.6,4.6) {};
\node [circle, fill, inner sep=1.8pt, opacity=0] (B4) at (6.5,4.7) {};

\node [circle, fill, inner sep=1.8pt, opacity=1] (a) at (4,8) {};
\node [circle, fill, inner sep=1.8pt, opacity=1] (b) at (2,9) {};
\node [circle, fill, inner sep=1.8pt, opacity=1] (c) at (3,11) {};
\node [circle, fill, inner sep=1.8pt, opacity=1] (d) at (5,10) {};

\node [circle, fill, inner sep=1.8pt, opacity=1] (e) at (6.8,7) {};
\node [circle, fill, inner sep=1.8pt, opacity=1] (ee) at (7.5,8.2) {};
\node [circle, fill, inner sep=1.8pt, opacity=1] (eee) at (8.5,9.3) {};
\node [circle, fill, inner sep=1.8pt, opacity=1] (f) at (10,10) {};
\node [circle, fill, inner sep=1.8pt, opacity=1] (g) at (12,13.5) {};
\node [circle, fill, inner sep=1.8pt, opacity=1] (gg) at (9.7,8.5) {};
\node [circle, fill, inner sep=1.8pt, opacity=1] (ggg) at (8.8,7.2) {};
\node [circle, fill, inner sep=1.8pt, opacity=1] (gggg) at (7.8,6.2) {};

\node [circle, fill, inner sep=1.8pt, opacity=1] (h) at (4.2,3.8) {};
\node [circle, fill, inner sep=1.8pt, opacity=1] (hh) at (3.2,2.8) {};
\node [circle, fill, inner sep=1.8pt, opacity=1] (hhh) at (2.3,1.5) {};
\node [circle, fill, inner sep=1.8pt, opacity=1] (i) at (2,0) {};
\node [circle, fill, inner sep=1.8pt, opacity=1] (ii) at (3.5,0.7) {};
\node [circle, fill, inner sep=1.8pt, opacity=1] (iii) at (4.5,1.8) {};
\node [circle, fill, inner sep=1.8pt, opacity=1] (iiii) at (5.2,3) {};

\node [circle, fill, inner sep=1.8pt, opacity=1] (j) at (7.5, 2.5) {};
\node [circle, fill, inner sep=1.8pt, opacity=1] (k) at (9.2,1) {};
\node [circle, fill, inner sep=1.8pt, opacity=1] (l) at (8.4, 3.2) {};

\draw [-stealth, black, thick] (B1) to [out=150, in=280] (a);
\draw [-stealth, black, thick] (a) to  (b);
\draw [-stealth, black, thick] (b) to  (c);
\draw [-stealth, black, thick] (c) to  (d);
\draw [-stealth, black, thick] (d) to  (a);
\draw [-stealth, black, thick] (a) to [out=350, in=90] (B1);

\draw [-stealth, black, thick] (B2) to [out=90, in=250] (e);
\draw [-stealth, black, thick] (e) to [out=70, in=230] (ee);
\draw [-stealth, black, thick] (ee) to [out=60, in=210] (eee);
\draw [-stealth, black, thick] (eee) to [out=50, in=190] (f);
\draw [-stealth, black, thick] (f) to [out=85, in=210] (g);
\draw [-stealth, black, thick] (g) to [out=270, in=35] (f);
\draw [-stealth, black, thick] (f) to [out=270, in=65] (gg);
\draw [-stealth, black, thick] (gg) to [out=250, in=50] (ggg);
\draw [-stealth, black, thick] (ggg) to [out=245, in=35] (gggg);
\draw [-stealth, black, thick] (gggg) to [out=225, in=10] (B2);

\draw [-stealth, black, thick] (B3) to [out=190, in=45] (h);
\draw [-stealth, black, thick] (h) to [out=205, in=55] (hh);
\draw [-stealth, black, thick] (hh) to [out=230, in=70] (hhh);
\draw [-stealth, black, thick] (hhh) to [out=255, in=80] (i);
\draw [-stealth, black, thick] (i) to [out=10, in=230] (ii);
\draw [-stealth, black, thick] (ii) to [out=30, in=240] (iii);
\draw [-stealth, black, thick] (iii) to [out=40, in=250] (iiii);
\draw [-stealth, black, thick] (iiii) to [out=70, in=270] (B3);

\draw [-stealth, black, thick] (B4) to [out=280, in=130] (j);
\draw [-stealth, black, thick] (j) to [out=300, in=155] (k);
\draw [-stealth, black, thick] (k) to [out=100, in=300] (l);
\draw [-stealth, black, thick] (l) to [out=120, in=350] (B4);


\node [circle, draw, thick] at (19,5) {$B$};
\node [circle, fill, inner sep=1.8pt, opacity=0] (BB1) at (18.6,5.4) {};
\node [circle, fill, inner sep=1.8pt, opacity=0] (BB2) at (19.4,5.4) {};
\node [circle, fill, inner sep=1.8pt, opacity=0] (BB3) at (18.6,4.6) {};
\node [circle, fill, inner sep=1.8pt, opacity=0] (BB4) at (19.5,4.7) {};

\node [circle, fill, inner sep=1.8pt, opacity=1] (a1) at (17,8) {};

\node [circle, fill, inner sep=1.8pt, opacity=1] (e1) at (19.8,7) {};
\node [circle, fill, inner sep=1.8pt, opacity=1] (ee1) at (20.5,8.2) {};
\node [circle, fill, inner sep=1.8pt, opacity=1] (eee1) at (21.5,9.3) {};
\node [circle, fill, inner sep=1.8pt, opacity=1] (f1) at (23,10) {};
\node [circle, fill, inner sep=1.8pt, opacity=1] (g1) at (25,13.5) {};
\node [circle, fill, inner sep=1.8pt, opacity=1] (gg1) at (22.7,8.5) {};
\node [circle, fill, inner sep=1.8pt, opacity=1] (ggg1) at (21.8,7.2) {};
\node [circle, fill, inner sep=1.8pt, opacity=1] (gggg1) at (20.8,6.2) {};

\node [circle, fill, inner sep=1.8pt, opacity=1] (h1) at (17.3,3.8) {};
\node [circle, fill, inner sep=1.8pt, opacity=1] (hh1) at (16.5,2.6) {};
\node [circle, fill, inner sep=1.8pt, opacity=1] (i1) at (16,1) {};
\node [circle, fill, inner sep=1.8pt, opacity=1] (ii1) at (17.4,2.0) {};
\node [circle, fill, inner sep=1.8pt, opacity=1] (iii1) at (18.3,3.2) {};

\node [circle, fill, inner sep=1.8pt, opacity=1] (j1) at (20.5, 2.5) {};
\node [circle, fill, inner sep=1.8pt, opacity=1] (k1) at (22.2,1) {};
\node [circle, fill, inner sep=1.8pt, opacity=1] (l1) at (21.4, 3.2) {};

\draw [-stealth, black, thick] (BB1) to [out=150, in=280] (a1);
\draw [-stealth, black, thick] (a1) to [out=350, in=90] (BB1);

\draw [-stealth, black, thick] (BB2) to [out=90, in=250] (e1);
\draw [-stealth, black, thick] (e1) to [out=70, in=230] (ee1);
\draw [-stealth, black, thick] (ee1) to [out=60, in=210] (eee1);
\draw [-stealth, black, thick] (eee1) to [out=50, in=190] (f1);
\draw [-stealth, black, thick] (f1) to [out=85, in=210] (g1);
\draw [-stealth, black, thick] (g1) to [out=270, in=35] (f1);
\draw [-stealth, black, thick] (f1) to [out=270, in=65] (gg1);
\draw [-stealth, black, thick] (gg1) to [out=250, in=50] (ggg1);
\draw [-stealth, black, thick] (ggg1) to [out=245, in=35] (gggg1);
\draw [-stealth, black, thick] (gggg1) to [out=225, in=10] (BB2);

\draw [-stealth, black, thick] (BB3) to [out=190, in=45] (h1);
\draw [-stealth, black, thick] (h1) to [out=225, in=65] (hh1);
\draw [-stealth, black, thick] (hh1) to [out=235, in=80] (i1);
\draw [-stealth, black, thick] (i1) to [out=10, in=230] (ii1);
\draw [-stealth, black, thick] (ii1) to [out=40, in=240] (iii1);
\draw [-stealth, black, thick] (iii1) to [out=70, in=270] (BB3);

\draw [-stealth, black, thick] (BB4) to [out=280, in=130] (j1);
\draw [-stealth, black, thick] (j1) to [out=300, in=155] (k1);
\draw [-stealth, black, thick] (k1) to [out=100, in=300] (l1);
\draw [-stealth, black, thick] (l1) to [out=120, in=350] (BB4);

\node at (5, -2) [coordinate, draw, fill=black, label=above: The graph $G$] {};
\node at (19, -2) [coordinate, draw, fill=black, label=above: The graph $H$] {};

\end{tikzpicture}
\end{center}
\caption{A example of $G$ and $H$, where dots denote vertices.}
\label{cons_graph_H}
\end{figure}
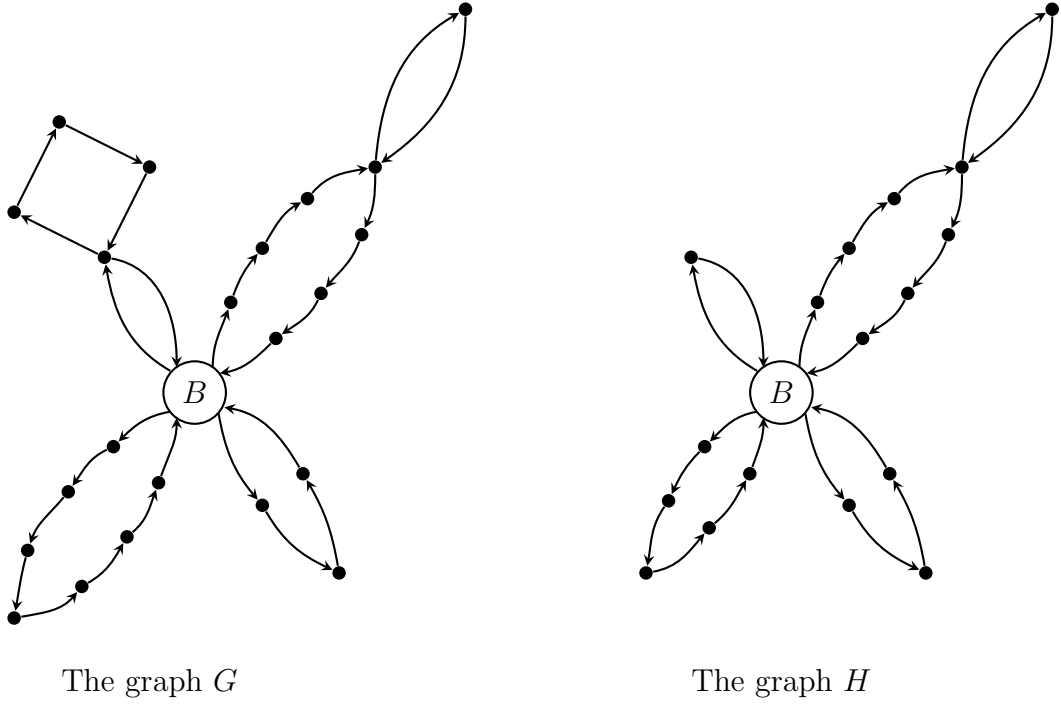

The following properties are immediate from the construction of $H$:
\begin{itemize}
\item[$(a)$] $H$ is a spoke graph. It consists of the central state $B$ and several spokes intersecting at $B$, where spokes in $H_1$ are regular spokes and spokes in $H_2\cup H_3$ are degenerate spokes.
\item[$(b)$] $H_1\cup H_3$ is a subgraph of $G$;
\item[$(c)$] If $\eta_1$ and $\eta_2$ are two different simple cycles at $B$, then $\vert\eta_1\vert \neq \vert\eta_2\vert.$
\end{itemize}

Now, define a one-block map $\Psi: \mathcal{V}(H)\rightarrow \mathcal{V}(G)$ as follows:
\begin{itemize}
\item For $v\in \mathcal{V}(H_1\cup H_3)$, let $\Psi(v)=v$;
\item 
For any $r\in S^{(1)}\setminus S^{(0)}$, choose a cycle $\widetilde{C}(r)$ in $G$ starting and ending with $B$  with no $B$ in its interior such that $\vert \widetilde{C}(r) \vert= \vert C(r) \vert$. Define
$$
\Psi(V(C(r))):= V(\widetilde{C}(r)).
$$
\end{itemize}
Note that for any two distinct vertices $v_1, v_2\in \mathcal{V}(H)$,
$\Psi(v_1)=\Psi(v_2)$ only if there exist $r_1, r_2\in S^{(1)}\setminus S^{(0)}$ with $r_1 \neq r_2$ such that $v_1\in V(C(r_1))$ and $v_2\in V(C(r_2))$, where $C(r_1)$ and $C(r_2)$ are constructed in Step $(B)$.

Let $\psi:\widehat{X_H}\rightarrow \widehat{X_G}$ be the sliding block code induced by $\Psi$ and define $Z: =\psi(\widehat{X_H})$. Note that any point $z\in Z$ is a concatenation of strings of the form
\begin{equation} \label{four_blocks}
Bu_1 u_2 \cdots u_k B,\quad \cdots v_{-3}v_{-2}v_{-1} B,\quad Bw_1w_2w_3\cdots,\quad \cdots i_{-2} i_{-1} i_0 i_1 i_2 \cdots
\end{equation}
where $u_j$'s, $v_j$'s, $w_{j}$'s and $i_j$'s are vertices in $G$ distinct from $B$. Thus, To show that $\psi$ is one-to-one, it suffices to show that any string in (\ref{four_blocks}) has a unique $\Psi$-pre-image, and we prove this by considering the following cases:
\begin{enumerate}
\item[(1)] Any allowed block of the form $Bu_1 u_2 \cdots u_k B$ in $\widehat{X_G}$ must be the $\Psi$-image of some block of the form $B x_1 x_2 \cdots x_k B$ with $x_i\in \mathcal{V}(H)$ for any $1\leq i\leq k$. Noting from Property $(c)$ that each $B x_1 x_2 \cdots x_k B$ is uniquely determined by its length, we conclude that the $\Psi$-pre-image of $Bu_1 u_2 \cdots u_k B$ is unique.
\item[(2)] For simplicity, among the infinite paths in (\ref{four_blocks}), we consider only those of the form $\cdots v_{-3}v_{-2}v_{-1} B$ in $\widehat{X_G}$. Such a  string must be the $\Psi$-image of some string of the form $\cdots x_{-3} x_{-2} x_{-1} B$ with $x_i\in \mathcal{V}(H_1)$. Since $\Psi$ is the identity map on $\mathcal{V}(H_1\cup H_3)$, $\cdots v_{-3}v_{-2}v_{-1} B$ has a unique $\Psi$-pre-image.
\end{enumerate}

Let $Z:=\psi(\widehat{X_H})$. Then $Z$ is an irreducible SFT because it is conjugate to
$\widehat{X_H}$. 
We now prove that $\phi\vert_Z$ is almost invertible and onto $Y$.
Note that by definition $\Phi\circ \Psi$ maps the central state $B$ to $1$ and maps all other vertices in $H$ to $0$.  So $\phi\circ \psi$ is the standard factor code on the
spoke graph $H$.

To see that $\phi\circ \psi$ is onto, first note that the image $(\phi\circ\psi)(\widehat{X_{H}})$ is a gap shift with gaps of the form
\begin{align*}
S' &:= S^{(0)} \cup (S^{(1)} \setminus S^{(0)}) \cup (S^{(2)}\setminus S^{(1)}) \\
&=S^{(0)} \cup S^{(1)} \cup S^{(2)}\\
&=S^{(1)} \cup S^{(2)}
\end{align*}
 where we use the fact that $S^{(0)}\subset S^{(1)}$ in the last equation.
Since $\cup_{i \in W} K_i  = \cup_{i \in T_1} K_i$, we have $S' =S$ where $S$ is such that $Y$ is an $S$-gap shifts. Therefore $\phi\circ\psi$ is onto.

 We now show that $\phi\circ \psi$ is finite-to-one. We first note from the construction of $H$ that for any $t\in S$, there is a unique cycle of length $t+1$ in $H$ starting and ending with $B$, whose interior does not contain $B$. Hence,
for any $t\in S$, there is a unique path in $H$ whose image under $\Phi\circ \Psi$ is $10^t1$. This implies that $\phi\circ \psi$ has no graph diamond and therefore it is finite-to-one.

Since the central state $B$ is the only vertex in $H$ whose $(\Phi\circ\Psi)$-image is $1$,
and since $\phi\circ\psi$ is a finite-to-one 1-block code on a 1-step SFT, its  degree is $1$ (by [LM95, Theorem 9.1.11 (3) and Proposition 9.1.12]) and therefore it is almost invertible.

Finally, since $\phi\circ\psi$ is almost invertible and onto Y, and $\psi$ is a conjugacy from $\widehat{X_H}$ to $Z$, we conclude that $\phi\vert_{Z}: Z\rightarrow Y$ is almost invertible and onto.
\medskip

{\em (2)} $\Rightarrow$ {\em (3)}: As we said in Section 2, any almost invertible factor code
on an irreducible SFT is finite-to-one [LM95, Proposition 9.2.2].


{\em (3)} $\Rightarrow$ {\em (1)}: Suppose that there is an irreducible SFT $Z\subset X$ such that $\phi\vert_{Z}$ is finite-to-one and onto. Let $k$ be the degree of $\phi\vert_{Z}$ and $L$ be the maximum length of words in a forbidden list of blocks from $X$ that defines $Z$. Then, there exist a word of the form $u:=0^{e_1}10^{e_2}1\cdots 10^{e_n}$ such that each $e_i \in S$, an integer $L\leq M\leq \vert u\vert$, and an index $1\leq j\leq \vert u \vert-M+1$ such that
the set
$$
E:=\{v_{[j, j+M-1]}:  v\in\mathcal{B}(Z), \Phi(v) =u\}
$$
has cardinality $k$.
Note that $u$ is a magic word and $u_{[j, j+M-1]}$ is the corresponding magic block.

For notational convenience, in the remainder of this proof, for any block $w$ with length $\vert u \vert $, we use the following notation:
\begin{align*}
\overline{w}:= w_{[1, j-1]}, \quad \widetilde{w}:=w_{[j, j+M-1]}, \quad \widehat{w}:=w_{[j+M, \vert u\vert]}
\end{align*}
where $u$, $j$ and $M$ are defined as above.

Denote elements in $E$ by $a^{(1)}, a^{(2)}, \cdots, a^{(k)}$ and for any $1\leq l \leq k$, define
$$
B^{(l)}:=\{v\in \mathcal{B}(Z): \Phi(v)=u, \widetilde{v}=a^{(l)}\} \quad \mbox{ and } \quad
R:=\cup_{1\leq l\leq k} B^{(l)}.
$$
Note that $R$ is the set of all $\phi\vert_{Z}$-pre-images of $u$.
By a higher block recoding similar to [LM95, Proposition 9.1.7], the following observation follows from [LM95, Proposition 9.1.9 (part 2)].
\medskip

\noindent{\em Observation 1: }
Let $u x u$ be a word in $\mathcal{B}(Y)$ and let $A: =\{z\in \mathcal{B}(Z): \Phi(z)= uxu\}$. Note that any element in $A$ is of the form $v^{(l)} w v^{(l')}$ where $1\leq l ,l' \leq k$ and $v^{(l)}\in B^{(l)}, v^{(l')}\in B^{(l')}$. Then, there exists a permutation $\tau=\tau_{uxu}$ of $\{1,2,\cdots k\}$ such that for any pair $(l, l')$, $v^{(l)} w v^{(l')}\in A$ for some $w$ only if $l'=\tau(l)$.
\medskip


For any $1\leq l\leq k$, define
\begin{align*}
F^{(l)}:=\{i\in T_1: &vV(\gamma_i^+ (C_i)^L \gamma_i^-)w \in \mathcal{B}(Z) \mbox{ for some $v\in B^{(l)}$ and some $w\in R$}
\}
\end{align*}
to be the index set of regular spokes that can follow some pre-images of $u$ in $B^{(l)}$ and precede some pre-images of $u$ in $R$.
We claim that for any $1\leq l\leq k$, $\{K_i: i\in F^{(l)}\}$ are pairwise disjoint and $\cup_{i\in F^{(l)}} K_i =\cup_{i\in T_1} K_i.$ We assume WLOG that $l=1$ in the following.


To show $\{K_i: i\in F^{(1)}\}$ are pairwise disjoint, we suppose to the contrary that there exists $f\in K_{i_1} \cap K_{i_2}$ for some $i_1, i_2\in F^{(1)}$ with $i_1\neq i_2$. Choose $n(f)= f (\bmod \ D)$ such that $n(f)\geq \max\{d_{i_1} L+m_{i_1}, d_{i_2} L+m_{i_2}\}$. Then, $n(f)\in S$ and according to the definition of $F^{(1)}$, there are $v, x\in B^{(1)}$, $w\in B^{(l_1)}$, $y\in B^{(l_2)}$ for some $1\leq l_1, l_2 \leq k$ such that
\begin{align}
\Phi(vV(\gamma_{i_1}^+ (C_{i_1})^{(n(f)-m_{i_1})/d_{i_1}}\gamma_{i_1}^- )w)&=\Phi(xV( \gamma_{i_2}^+ (C_{i_2})^{(n(f)-m_{i_2})/d_{i_2}}\gamma_{i_2}^- )y)=u10^{n(f)}1 u,  \notag \\
\mbox{ and } \qquad \widetilde{v}=\widetilde{x}=a^{(1)}&, \quad \widetilde{w}=a^{(l_1)}, \quad \widetilde{y}=a^{(l_2)}. \notag
\end{align}
Then, we infer from Observation 1 that $l_1=l_2$ and therefore $\widetilde{w}=\widetilde{y}=a^{(l_1)}$.
Now, the two words
\begin{align} \label{two_blocks_per}
\widetilde{v} \widehat{v}V( \gamma_{i_1}^+ (C_{i_1})^{(n(f)-m_{i_1})/d_{i_1}}\gamma_{i_1}^-) \overline{w} \widetilde{w} \quad \mbox{and} \quad \widetilde{x} \widehat{x} V(\gamma_{i_2}^+ (C_{i_2})^{(n(f)-m_{i_2})/d_{i_2}}\gamma_{i_2}^-) \overline{y} \widetilde{y}
\end{align}
are both $\phi\vert_{Z}$-pre-images of $\widetilde{u} \widehat{u}1 0^{n(f)}1 \overline{u}\widetilde{u}$, and they both start with $a^{(1)}$ and end with $a^{(l_1)}$. Since $a^{(1)}$ and $a^{(l_1)}$ both have length $M$, which is no less than $L$, we deduce that the two words in (\ref{two_blocks_per}) can be extended to a point diamond, contradicting the fact that $\phi\vert_Z$ is finite-to-one. 

To show $\cup_{i\in F^{(1)}} K_i =\cup_{i\in T_1} K_i,$
assume to the contrary that there is a $g\in \cup_{i\in T_1} K_i$ but $g\notin \cup_{i\in F^{(1)}} K_i$. Choose $n(g):= g\  (\bmod \ D)$ such that $n(g) > \max\{ d_i: i\in T_0\} $ and $n(g)\geq \max\{d_i L + m_i: i\in T_1\}$. Then, $n(g)\in S$ and
we deduce from the definition of $F^{(1)}$ that the set
$$
Q:=\{z_{[j, j+M-1]}: z\in \mathcal{B}(Z),  \Phi(z)= u 10^{n(g)}1 u\}
$$
does not contain $a^{(1)}$. Noting that $Q\subset \{a^{(1)}, a^{(2)}, \cdots, a^{(k)}\}$ since $u$ is a magic word, the cardinality of $Q$ is at most $k-1$. This contradicts the fact that $\phi\vert_{Z}$ has degree $k$, and therefore $\cup_{i\in F^{(1)}} K_i =\cup_{i\in T_1} K_i$.

Now let $W= F^{(1)}$. Then, we immediately infer from above that $W$ is the desired set and therefore complete the proof.
\end{proof}

\begin{rem}
Our proof indeed shows that conditions $(2)$ and $(3)$ in Theorem \ref{spoke_three_eq} are equivalent for any 1-block factor code with an unambiguous symbol defined on a 1-step SFT.
\end{rem}

\begin{rem} \label{conti_example}
Let $G$ be the graph in Figure \ref{single_spoke} where $B$ is the central state. Let $\phi$ be the standard factor code on $G$. Then, one verifies that $\Phi$ (which generates $\phi$) has no graph diamond and so $\phi$ is finite-to-one; on the other hand, $\phi$ is not one-to-one: both $(V_1V_2)^\infty$ and $(V_2V_1)^{\infty}$ are  preimages of $0^\infty$. In this case, there is no subshift $Z\subset \widehat{X_G}$ such that $\phi\vert_Z$ is one-to-one and onto (see discussion in Remark \ref{f-to-1_but_not_1-to-1}).

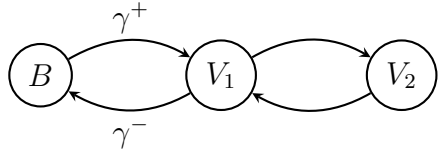
\begin{figure}[H]
\begin{center}
\begin{tikzpicture}[scale=0.6]
\node [circle, draw, thick] (a) at (7,1) {$B$};
\node [circle, draw, thick] (b) at (11,1) {$V_1$};
\node [circle, draw, thick] (c) at (15,1) {$V_2$};
\draw [-stealth, black, thick] (a) to [out=30, in=150] (b) ;
\draw [-stealth, black, thick] (b) to [out=210, in=330] (a);
\draw [-stealth, black, thick] (b) to [out=30, in=150] (c);
\draw [-stealth, black, thick] (c) to [out=210, in=330] (b);
\node at (9, 1.7) [coordinate, draw, fill=black, label=above: $\gamma^+$] {};
\node at (9, 0.3) [coordinate, draw, fill=black, label=below: $\gamma^-$] {};
\end{tikzpicture}
\end{center}
\caption{The graph $G$, which is a representation of $X_{\mathcal{F}}$.}
\label{single_spoke}
\end{figure}

\end{rem}

\begin{exmp}
Let $G$ be the 3-spoke graph defined by
$$d_1=d_3=6, \quad d_2=3, \quad m_1=m_2=1, \quad m_3=4$$
and $\phi$ be the standard factor code on $G$.
Then, $T_0=\emptyset$, $T_1=\{1,2,3\}$,  $D=l.c.m.(d_1,d_2,d_3)=6$ and
$$
K_1=\{1\}, \quad K_2=\{1,4\}, \quad K_3=\{4\}.
$$
Here the image $Y$ of $\phi$ is an $S$-gap shift with
$$S=\{n\in \mathbb{Z}_{\geq 0}: n=1 \bmod 3\}.$$

There are two ways to choose $W$:

(1) $W=\{1,3\}$. It can be readily checked that $\cup_{i\in W} K_i = \cup_{i\in T_1} K_i$ and $K_1 \cap K_3 = \emptyset$.  So, by Theorem \ref{spoke_three_eq} there is an SFT $Z \subset
\widehat{X_G}$ such that $\phi_Z$ is finite-to-one and onto $Y$.  In this case,
the proof chooses
$Z$ to be $\widehat{X_{U_1 \cup U_3}}$.

(2) $W=\{2\}$. Here, the proof of Theorem \ref{spoke_three_eq} chooses $Z$ to be
$\widehat{X_{U_2}}$.

This shows that there are two irreducible Markov measures $\nu_1$ and $\nu_2$, with $\nu_1$ supported on $\widehat{X_{U_1\cup U_3}}$ and $\nu_2$ ssupported on $\widehat{X_{U_2}}$, that both achieve the capacity of the channel given by the standard factor code on ${G}$.  Using $\nu_1$ and $\nu_2$, one can construct a fully supported irreducible Markov measure on $\widehat{X_G}$ that achieves capacity.
\end{exmp}

\begin{exmp}
Let $G$ be the 4-spoke graph defined by
$$
m_1=m_2=m_3=1, \quad m_4=10, \quad d_1=2,\quad d_2=3, \quad d_3=4, \quad d_4=6
$$
and $\phi$ be the standard factor code on $G$.
Then, $$T_0=\emptyset, \ \ T_1=\{1,2,3,4\}, \ \  D=l.c.m. (d_1, d_2,d_3,d_4)= 12$$ and
$$
K_1=\{1,3,5,7,9,11\}, \quad K_2=\{1,4,7,10\}, \quad K_3=\{1,5,9\}, \quad K_4=\{4,10\}.
$$
Let $Y$ be the image of $\phi$.
Since $K_1 \cap K_4 =\emptyset$ and $K_1 \cup K_4 =\cup_{i\in T_1} K_i$, it follows from Theorem~\ref{spoke_three_eq} that there is an SFT $Z\subset \widehat{X_G}$ such that $\phi\vert_Z$ is finite-to-one and onto $Y$. Note that in this example, we cannot simply choose $H$ in the proof of Theorem \ref{spoke_three_eq} to be the graph obtained from $G$ by deleting $U_2$ and $U_3$.
This is because $10^41$ is allowed in $Y$, but not allowed in  $\phi(\widehat{X_{U_1 \cup U_4}})$: the only $\Phi$-preimage of $10^41$ is $V(\gamma_2^+ C_2 \gamma_2^-)$ and it comes only from the spoke $U_2$. Instead, we  let $H$ be the graph obtained from $G$ by deleting $U_2$ and $U_3$, and then adding to $H$ a cycle of length $5$ starting and ending with $B$. Then, according to the proof of Theorem \ref{spoke_three_eq}, ${\widehat{X_H}}$ is conjugate to some SFT $Z\subset \widehat{X_G}$ and $\phi\vert_{Z}$ is finite-to-one and onto $Y$.
\end{exmp}

\begin{exmp} An example for which the conditions in Theorem \ref{spoke_three_eq} are not satisfied is given in Section 3 of [MPW84]. Here, $G$ is the 4-spoke graph defined by
$$
m_1=m_2=1, \ \ m_3=2, \ \ m_4=6, \ \ d_1=2, \ \ d_2=3,  \ \ d_3=6, \ \ d_4=6.
$$
Let $\phi$ be the standard factor code on $G$. It was shown in [MPW84] that for this $\phi$, P3 is not true and therefore P2 is not true.
\end{exmp}

\section{Conjecture: P2 and P3 are equivalent for Standard Factor Codes on Spoke Graphs}
\label{Conjecture}
Having characterized the condition under which P2 is satisfied for standard factor codes on spoke graphs, we now turn to the question whether P2 is equivalent to P3 for these codes. Recall from Proposition \ref{Fto1} that P2 always implies P3 for a general factor code. For the converse, we have the following.

\begin{conj} \label{(3)_implies_(2)}
Let $G$ be a spoke graph and $\phi$ be the standard factor code on $G$. Then P3 implies P2.
\end{conj}

\begin{rem}
It will be shown that if P3 holds, i.e.,
there is a $k$-th order Markov measure $\nu$ on $\widehat{X_G}$ s.t. $\phi^*(\nu)=\mu_0$, the mme on $Y$, then $\nu(V(C_i) \vert V((C_i)^k)) = Q^{-d_i},$
where $C_i$ is the cycle (disjoint from $B$) on the spoke $U_i$, $Q=e^{h_{top}(Y)}$, $d_i$ is the length of $C_i$. This is Part $(a)$ of the proof of Proposition \ref{Combined_properties} (see Equation (\ref{one})). Hence, the $\nu$-weight-per-symbol of each such $V((C_i)^\infty)$ is a constant $Q^{-1}$. If it is true that the weight-per-symbol of each of the periodic points $V((\gamma_i^+\gamma_i^-)^\infty)$ is also $Q^{-1}$, then one would have Condition 4 of Proposition \ref{Equiv_nu} and P2 would be true.  It may be that there is another Markov measure $\nu'$
on $\widehat{X_G}$ s.t. $\phi^*(\nu')=\mu_0$ such that this condition is satisfied.

\end{rem}

In the remainder of this section, we will prove some special cases of Conjecture \ref{(3)_implies_(2)}. To this end, we begin with some lemmas.

\begin{lem}(Consequence of Strong Form of Chinese Remainder Theorem) \label{Chinese_remainder}

\noindent Let $k$ be a positive integer. If for any $1\leq i< j\leq k$, there exists $x_{i,j}$ s.t. $x_{i,j}=a_i \ (\bmod\ d_i)$ and $x_{i,j}=a_j \ (\bmod\ d_j)$, then there exists $x$ such that $x=a_l \ (\bmod\ d_l)$ for any $1\leq l\leq k$.

\end{lem}

\begin{proof}
For any $1\leq i<j\leq k$, let $g_{i,j}=gcd(d_i, d_j)$. Then $g_{i,j}$ divides $x_{i,j}-a_i$ and $x_{i,j}-a_j$ so $g_{i,j}$ divides $a_i-a_j$. Hence, the generalized Chinese Remainder Theorem~[Le56, Theorem 3-12] asserts that there is a common solution to $x=a_i \ (\bmod\ d_i), i=1,2,\cdots, k$.
\end{proof}

\begin{lem} \label{Combined_properties}
Let $\nu$ be a $k$-th order Markov measure on $\widehat{X_G}$ such that P3 holds. Define $\Pi_i := \nu(V(\gamma_i^+(C_i)^{Dk/d_i}) \vert B)$, $P:= \{i\in T_1: \Pi_i>0\}$ and $R_j:= \{i\in T_1: j\in K_i\}=\{i\in T_1: d_i \mbox{ divides } j-m_i \}$. Then
\begin{enumerate}
\item[(a)] For each $0\leq j<D$,
\begin{align} \label{one}
Q^{-Dk} = \sum_{i \in R_j\cap P} \Pi_i Q^{m_i + 1}(1 - Q^{-d_i})
\end{align}
where $Q:=e^{h_{top}(Y)}$;
\item[(b)] $\cup_{i\in T_1\setminus P} K_i \subset \cup_{i\in P} K_i$;
\item[(c)] For each pair $j, j'$, 
if $R_{j'}\cap P \subset R_j\cap P$, then $R_{j'} \cap P = R_{j}\cap P$.
\end{enumerate}
\end{lem}

\begin{proof}
Fix a congruence class $0\le j < D$ and let $\mu_0$ be the mme on $Y$.

Since $\phi(\nu) = \mu_0$, for all $n\ge \max_{i\in T} m_i/D$,
$$
\mu_0(10^{Dk + j + Dn}1) = \sum_{i \in R_j} \nu(V(\gamma_i^+ (C_i)^{(1/d_i)(Dk + j +  Dn - m_i)}\gamma_i^-)).
$$
Let
$$
Q_i := \nu(V(C_i) \vert V(\gamma_i^+(C_i)^{(1/d_i)Dk})).
$$
Since $\mu(1) = \nu(B)$, using the formula for the mme $\mu$ we have
\begin{align*}
Q^{-(Dk + j + Dn + 1)} &= \sum_{i \in R_j} \Pi_i (Q_i^{1/d_i})^{-m_i}
(Q_i^{1/d_i})^{Dn+j}(1-Q_i) \\
&= \sum_{i \in R_j\cap P} \Pi_i (Q_i^{1/d_i})^{-m_i}
(Q_i^{1/d_i})^{Dn+j}(1-Q_i)
\end{align*}
and so
$$
Q^{-(Dk  + 1)} = \sum_{i \in R_j\cap P} \Pi_i (Q_i^{1/d_i})^{-m_i}
\left(\frac{Q_i^{1/d_i}}{Q^{-1}}\right)^{Dn+j}(1-Q_i).
$$
Letting $n \rightarrow \infty$, we have for all $i \in R_j \cap P$,
\begin{align} \label{Q_and_Q_i}
\frac{Q_i^{1/d_i}}{Q^{-1}} = 1.
\end{align}
This yields equation (\ref{one}) and proves $(a)$.

Since $Y$ is a gap shift, the mme on $Y$ is fully supported and so gives positive measure to each allowed gap. Thus,
 $\cup_{i \in T_1 \setminus P} K_i \subset \cup_{i \in P} K_i$, proving $(b)$.

To see $(c)$, we first derive from (\ref{one}) that
$$
\sum_{i \in R_{j'}\cap P} \Pi_iQ^{m_i+1}(1 - Q^{-d_i}) = Q^{-Dk} = \sum_{i \in R_j\cap P} \Pi_iQ^{m_i+1}(1 - Q^{-d_i}).
$$
Thus,
$$
\sum_{i \in (R_{j}\cap P) \setminus (R_{j'} \cap P)}  \Pi_iQ^{m_i + 1}(1 - Q^{-d_i}) =0
$$
which immediately implies $ (R_{j}\cap P) \setminus (R_{j'} \cap P)  = \emptyset$.
\end{proof}

\begin{lem} \label{induct}
Let $P$ be defined as in Lemma~\ref{Combined_properties} and $i_1, i_2\in P$ with $K_{i_1}\cap K_{i_2}\neq \emptyset$. Then, for any $j\in K_{i_1} \setminus K_{i_2}$, there exists $i_3\in P$ such that
\begin{enumerate}
\item[(1)] $j\in K_{i_3}$;
\item[(2)] $K_{i_2}\cap K_{i_3}= \emptyset.$
\end{enumerate}
\end{lem}
\begin{proof}
For notational convenience, we rewrite $j$ by $j_1$ and define
${\small S(i_1, i_2, j_1):=(R_{j_1} \cap P) \setminus \{i_1, i_2\}}$, where $R_{j_1}$ is defined in Lemma \ref{Combined_properties}.

We first show that $S(i_1, i_2, j_1)\neq \emptyset$. Suppose to the contrary that $S(i_1, i_2, j_1)= \emptyset$. Then, $R_{j_1}\cap P=\{i_1\}$. Since $K_{i_1} \cap K_{i_2}\neq \emptyset$, there exists $j_2\in K_{i_1} \cap K_{i_2}$ and therefore $R_{j_2}\cap P \supset \{i_1, i_2\}$. Hence, $R_{j_1}\cap P\subsetneq R_{j_2} \cap P$, contradicting Lemma \ref{Combined_properties} $(c)$.

We then claim that there exists $i_3\in S(i_1, i_2, j_1)$ such that $K_{i_2} \cap K_{i_3}=\emptyset$. If not, then
$$
K_{i_2} \cap K_i\neq \emptyset \quad  \mbox{ for any }i\in S(i_1, i_2, j_1).
$$
Recalling that $j_2\in K_{i_1}\cap K_{i_2}$ and $j_1\in K_{i_1} \cap (\cap_{i\in S(i_1, i_2, j_1)}K_i)$, we derive from Lemma \ref{Chinese_remainder} that there exists $j_4\in K_{i_1} \cap K_{i_2} \cap (\cap_{i\in S(i_1, i_2, j_1)} K_i).$ Hence,
$$R_{j_1}\cap P =\{i_1\}\cup S(i_1, i_2, j_1) \subsetneq \{i_1, i_2\}\cup S(i_1, i_2, j_1) \subset R_{j_4}\cap P,$$
contradicting Lemma \ref{Combined_properties} $(c)$.
\end{proof}

With these lemmas in hand, we prove the following.
\begin{pr}\label{partialconverse}
Let $G$ be a spoke graph, $\phi$ be the standard factor code on $G$ and $P$ be defined as in Lemma \ref{Combined_properties}. If there is a stationary Markov measure $\nu$ on $\widehat{X_G}$ s.t. $\phi^*(\nu)=\mu_0$, the mme of the output $Y$,
then there is an SFT $Z\subset \widehat{X_G}$ such that $\phi\vert_Z$ is finite-to-one and onto $Y$ if any of the following hold:
\begin{enumerate}
\item[(a)] $\cap_{i\in P} K_i \neq \emptyset$ (in particular, this holds when $m_i = 1$ for all $i$ or the $\{d_i\}$ are pairwise co-prime (by the Chinese Remainder Theorem));
\item[(b)] For any $i_1, i_2\in P$, $K_{i_1}\cap K_{i_2} \neq \emptyset$;
\item[(c)] There are subsets $E_1$ and $E_2$ of $P$ such that $\{K_i: i\in E_1\}$ and $\{K_i: i\in E_2\}$ are both pairwise disjoint and $\cup_{i\in E_1\cup E_2} K_i =\cup_{i\in T_1} K_i$. In particular, this condition is satisfied if there are only two distinct $d_i$'s;
\item[(d)] $\vert P\vert \leq 5$.
\end{enumerate}
\end{pr}

\begin{proof}
According to Theorem \ref{spoke_three_eq}, it suffices to show that there is a $W\subset T_1$ such that $\cup_{i\in W} K_i = \cup_{i\in T_1} K_i$ and $\{K_i: i\in W\}$ are pairwise disjoint.

\textbf{Proof of $(a)$: }
Let $A:= \cup_{i\in P} K_i$. Note that $P \ne \emptyset$ by the existence of $\nu$.
Let $j\in \cap_{i\in P} K_i$. Apply Lemma \ref{Combined_properties}$(c)$ to this $j$ and an arbitrary $j' \in A$ to get that  for all $i \in P$, $i \in \cap_{j'\in A} R_{j'}$ and so each
$K_i = A$. By Lemma \ref{Combined_properties} $(b)$, $A = \cup_{i \in T_1} K_i$.
Hence, $W$ can be taken to consist of only one element, namely any element of $P$.
\medskip

\textbf{Proof of $(b)$: } Since $K_i$'s are pairwise intersecting, an application of Lemma \ref{Chinese_remainder} to $\{K_i: i\in P\}$ implies that $\cap_{i\in P} K_i\neq \emptyset$ which is Case $(a)$.
\medskip

\textbf{Proof of $(c)$: }
We assume WLOG that the $K_i$'s are distinct.
Denote
\begin{align}
F:= \{i\in E_1: K_i\cap K_{i'} = \emptyset \mbox{ for all } i'\in E_2\}.
\end{align}
We claim that for any $i\in E_1 \setminus F$, $K_i\subset \cup_{i'\in E_2} K_{i'}$. To see this, assume to the contrary that there are $i_1 \in E_1 \setminus F$ and $j\in K_{i_1}$ such that $j \notin \cup_{ i' \in E_2} K_{i'}$. Recalling that $K_{i_1} \cap K_{i_2} =\emptyset$ for $i_1, i_2\in E_1$ with $i_1\neq i_2$, we have $R_{j}\cap P=\{i_1\}$. On the other hand, $i_1 \in E_1 \setminus F$ implies that there exists $j'\in K_{i_1}$ and $i_3\in E_2$ such that $j'\in K_{i_1} \cap K_{i_3}$. Hence, $R_{j'}\cap P\supset \{i_1, i_3\} \supsetneq \{i_1\}=R_{j} \cap P$, contradicting Lemma \ref{Combined_properties} $(c)$.

Now let $W:= F \cup E_2$. Clearly $\{K_i: i\in W\}$ are pairwise disjoint by the definition of $F$. Since $K_i\subset \cup_{i'\in E_2} K_{i'}$ for any  $i\in E_1 \setminus F$, $\cup_{i\in W} K_i=\cup_{i\in P} K_i=\cup_{i\in T_1} K_i$, proving $(c)$.
\medskip

\textbf{Proof of $(d)$:}
By adding repeated spokes (for which the choice of the set $W$ is not affected), we can regard the cases $\vert P\vert<5$ as special cases of $\vert P \vert=5$. Hence, we assume $\vert P \vert=5$ in the following.

Let $P=\{1,2,3,4,5\}$. A pair $i,i'\in P$ is called an {\em intersecting pair} if $K_i \ne K_{i'}$ and $K_i \cap K_{i'} \ne \emptyset$. We consider the following cases.

\noindent\textbf{\underline{Case 1}: For any intersecting pair $i, i'\in P$, either $K_{i}\subset K_{i'}$ or $K_{i'}\subset K_{i}$}.

In this case, we define a partial order $\preccurlyeq$ in the following way:
if $i,i'$ is an intersecting pair and $K_i \subset K_{i'}$, then $K_i\preccurlyeq K_{i'}$; if $i ,i'$ is not a intersecting pair, then $K_i$ and $K_{i'}$ are incomparable.

The partial order $\preccurlyeq$ partitions the set $\{K_i: i\in P\}$ into several classes such that
\begin{enumerate}
\item each class is a chain with a unique maximal element (under $\preccurlyeq$);
\item if $K_i$ and $K_{i'}$ are from different classes, then $K_{i} \cap K_{i'}=\emptyset$.
\end{enumerate}
Hence, letting $W$ be the indices of all the maximal elements, we have $\{K_i: i\in W\}$ are pairwise disjoint and $\cup_{i\in W} K_i = \cup_{i\in P} K_i =\cup_{i\in T_1} K_i$.

\noindent\textbf{\underline{Case 2}: There exists an intersecting pair $i, i'\in P$ such that both $K_{i} \nsubseteq K_{i'}$ and $K_{i'} \nsubseteq K_{i}$}.

First note that in this case, we necessarily have $d_i\neq d_{i'}$.
We may assume that $i=1$, $i'=2$ and $\displaystyle\frac{l.c.m.(d_1, d_2)}{d_1} \geq 3$. Let $j_1\in K_1\cap K_2$, $j_2\triangleq (j_1+d_1) \bmod (l.c.m.(d_1, d_2))$ and $j_3\triangleq (j_2+ d_1) \bmod (l.c.m.(d_1, d_2))$ where $0\leq j_2,j_3 < l.c.m.(d_1, d_2)$.
Then $j_2, j_3\in K_1\setminus K_2$. Furthermore, there is also a $j_4\in K_2 \setminus K_1$. Applying Lemma \ref{induct} to $j_2, j_3, j_4$, we deduce that there exist $l_1, l_2, l_3\in \{3,4,5\}$ such that
\begin{align}
j_2\in K_1 \cap K_{l_1}, \quad K_{l_1} \cap K_2 =\emptyset, \label{disjointness_K_2} \\
j_3\in K_1 \cap K_{l_2}, \quad K_{l_2} \cap K_2 =\emptyset, \notag \\
j_4\in K_2 \cap K_{l_3}, \quad K_{l_3} \cap K_1 =\emptyset. \label{K_1_K_5}
\end{align}
Note that we necessarily have $l_3\neq l_1$ and $l_3\neq l_2$. We now claim that $l_1\neq l_2$. To see this, suppose that $l_1=l_2=l$. Then $j_2\in K_l, j_3\in K_l$. Since $ j_3=(j_2+d_1) \bmod (l.c.m.(d_1, d_2))$ and $j_2\in K_1, j_3\in K_1$, we have $K_l\subset K_1$. Hence, $j_1\in K_1 \cap K_2 \cap K_l$, contradicting the fact that $K_2\cap K_l=\emptyset$. (See Figure \ref{figure3}, where for any $r,s$, a $\bullet$ ({\em resp. a $\times$}) on the $(r,s)$ position means that  $j_r\in K_s$ ({\em resp.} $j_r\notin K_s$)).


\begin{figure}[H]
$$
\begin{array}{cccccc}

j_4 & \times & \bullet & \times \\
j_3 & \bullet & \times & \bullet  \\

j_2 & \bullet & \times & \bullet\\

j_1 & \bullet & \bullet & \times  \\
\hline
(i,j) & K_1 & K_2 & K_{l} & K_4 & K_5\\
\end{array}
$$
\caption{Relationship between $K_1, K_2$ and $K_l$ if $l_1=l_2=l$.}
\label{figure3}
\end{figure}
\medskip


Hence, $l_1, l_2$ and $l_3$ are distinct. We may assume that $l_1=3, l_2=4$ and $l_3=5$. The current relation between $\{K_1, K_2 ,K_3, K_4,K_5\}$ is given in Figure \ref{figure4}, where $?$ means that whether this position is $\bullet$ or $\times$ is unknown up to now.
\begin{figure}[H] \label{figure4}
$$
\begin{array}{cccccc}

j_4 & \times & \bullet & \times & \times & \bullet \\

j_3 & \bullet & \times & ?_1 & \bullet & \times  \\

j_2 & \bullet & \times & \bullet & ?_2 & \times \\

j_1 & \bullet & \bullet & \times & \times & \times \\
\hline
(i,j) & K_1 & K_2 & K_3 & K_4 & K_5\\

\end{array}
$$
\caption{Relationship between $K_1, K_2 ,K_3, K_4, K_5$ with some unknowns.}
\end{figure}

We then claim that $j_3\notin K_3$ and $j_2\notin K_4$ (i.e., $?_1=?_2=\times$ in Figure \ref{figure4}).
To verify this claim, assume WLOG that $j_3\in K_3$. Then $j_2\in K_1\cap K_3$ and $j_3\in K_1 \cap K_3$. Noting that $j_3-j_2=d_1  \bmod (l.c.m.(d_1, d_2))$, we must have $K_3\supset K_1$, which contradicts the fact that $j_1\in K_1 \setminus K_3$. Hence, $j_3\notin K_3$. A similar argument shows that $j_2\notin K_4$, proving the claim.
\medskip



Now the relationship between $\{K_1, K_2, K_3, K_4, K_5\}$ is partially characterized in Figure \ref{figure5}.
\begin{figure}[H]
$$
\begin{array}{cccccc}

j_4 & \times & \bullet & \times & \times & \bullet \\

j_3 & \bullet & \times & \times & \bullet & \times  \\

j_2 & \bullet & \times & \bullet & \times & \times \\

j_1 & \bullet & \bullet & \times & \times & \times \\
\hline
(i,j) & K_1 & K_2 & K_3 & K_4 & K_5\\

\end{array}
$$
\caption{Relationship between $K_1, K_2 ,K_3, K_4, K_5$.}
\label{figure5}
\end{figure}

We then claim that $K_3\cap K_4=\emptyset$. To see this, suppose to the contrary that there is a $j_5\in K_3\cap K_4$. Since $j_2\in K_1\cap K_3$, $j_3\in K_1\cap K_4$, we infer from Lemma~\ref{Chinese_remainder} that there is a $j_6\in K_1\cap K_3 \cap K_4$, contradicting Lemma~\ref{Combined_properties} $(c)$.

Now let $E_1:=\{1,5\}, E_2:=\{2,3,4\}$. Since $\{K_i: i\in E_1\}$ and $\{K_i: i\in E_2\}$ are both pairwise disjoint, the desired result follows from Part (c).
\end{proof}

\begin{rem}
When $\vert P \vert \le 4$, by carefully going through a similar argument as in the proof of $(d)$, one can show that for any $i\neq j\in P$,  $K_i \cap K_j=\emptyset$ 
or $K_i\subset K_j$ or $K_j\subset K_i$.
\end{rem}

\section{Standard Factor Codes Defined on Another Class of Graphs} \label{two_cycles}
We believe that our approach in the proof of Theorem \ref{spoke_three_eq} also works for more general graphs. Note that for a graph $G$ with one (regular) spoke, Theorem \ref{spoke_three_eq} implies that P2 always holds. In this section, as an example we show that P2 also holds for standard factor codes defined on a more general class of graphs. To be specific, let $G$ be a graph which consists of a central state $B$, a simple path $\gamma^+$ from $B$ to $B'\neq B$, a simple path $\gamma^-$ from $B'$ to $B$, and two simple cycles $C_1$ and $C_2$ including $B'$ s.t.
\begin{enumerate}
\item[(a)] $\vert C_i\vert>0$ for $i=1,2$;
\item[(b)] $\gamma^+$ and $\gamma^-$ only intersect at $B$ and $B'$;
\item[(c)] $\gamma^+, \gamma^-, C_1$ and $C_2$ share the vertex $B'$ and there is no other common vertex among $\gamma^+, \gamma^-, C_1$ and $C_2$.
\end{enumerate}
Here, we implicitly assume that $\gamma^+\neq \emptyset$ and $\gamma^-\neq \emptyset$.

Just as in Section \ref{Sec_Spoke}, a standard factor code $\phi$ on $G$ is induced by a one-block map $\Phi: \mathcal{V}(G) \rightarrow \{0,1\}$ that maps the central state $B$ to $1$ and any other vertex to $0$.

Let $Y$ be the image of $\phi$. We have the following.

\begin{pr} \label{two_cycle_one_spoke}
Let $G$ be the graph defined above and $\phi$ be the standard factor code on $G$. Then, there is an SFT $Z\subset \widehat{X_{G}}$ s.t.
$\phi\vert_Z$ is finite-to-one and onto $Y$.
\end{pr}

We need the following lemma.

\begin{lem} \label{num_thy_lem}
Suppose $d_1, d_2$ are two positive integers. Let
\begin{align*}
E\:&=\{n\in \mathbb{Z}_{\geq 0}: n=s \cdot d_1+t \cdot d_2, \quad s,t\in \mathbb{Z}_{\geq 0}\}, \\
u:&=\frac{l.c.m.(d_1, d_2)}{d_2}.
\end{align*}
Then for any $n\in E$, the equation
\begin{align} \label{cons_eqn}
x\cdot d_1+ y\cdot d_2=n \qquad
s.t. \quad x,y\in \mathbb{Z}_{\geq 0}, 0\leq y< u
\end{align}
has a unique solution.
\end{lem}
\begin{proof}
We first show that (\ref{cons_eqn}) has a solution. Suppose $n=s\cdot d_1+t\cdot d_2$ for some $s,t\in \mathbb{Z}_{\geq 0}$. If $t< u$, then $x=s, y=t$ is a solution to (\ref{cons_eqn}); otherwise, if $t\geq u$, then there exist nonnegative integers $k,r$ with $0\leq r< u$ such that $t=k u +r$. Hence, we have
\begin{align}
n&= s \cdot d_1 + t \cdot d_2 \notag \\
&=s\cdot d_1+(ku+r) d_2 \notag \\
&= s \cdot d_1 + k \cdot (l.c.m.(d_1 ,d_2))+r d_2 \notag \\
&= \left(s+ k\cdot \frac{l.c.m. (d_1, d_2)}{d_1}\right) d_1 + r d_2 \label{new_ex_1},
\end{align}
Since $d_1 \mid l.c.m. (d_1, d_2)$ and $0\leq r<u$, we conclude from (\ref{new_ex_1}) that $x=s+ k\cdot \frac{l.c.m. (d_1, d_2)}{d_1}, y=r$ is a solution to (\ref{cons_eqn}).

We now prove that (\ref{cons_eqn}) has no more than one solution. Suppose to the contrary that there exist two different pairs of integers $(x_1, y_1)$ and $(x_2, y_2)$ that satisfy (\ref{cons_eqn}) and WLOG $y_1<y_2$. 
Now we have $
x_1 \cdot d_1 + y_1 \cdot d_2 = x_2 \cdot d_1 + y_2 \cdot d_2=n,
$
which implies
$ (y_2-y_1) d_2 = (x_1-x_2) d_1$. Hence, $d_1 \mid (y_2-y_1) d_2$ and it follows that
\begin{align} \label{con_eqn_1}
(y_2-y_1) d_2\geq l.c.m. (d_1, d_2)
\end{align}
since $d_2 \mid (y_2-y_1) d_2$. On the other hand, recalling that $y_1, y_2<u$, we have $y_2-y_1<u$ and
$$
(y_2-y_1) d_2< u\cdot d_2 = \frac{l.c.m.(d_1, d_2)}{d_2}\cdot d_2 = l.c.m. (d_1, d_2),
$$
contradicting (\ref{con_eqn_1}).
\end{proof}

\noindent{\bf Proof of Proposition {\ref{two_cycle_one_spoke}}:}
We first note that the image $Y$ of $\phi$ is a gap shift with gap set
 \begin{align*}
 S:= \{n \in \mathbb{Z}_{\geq 0}: n=m+s \cdot d_1 +t\cdot d_2 \mbox{ with }  s,t \in \mathbb{Z}_{\geq 0}\},
\end{align*}
 where $m = |\gamma^+| + |\gamma^-|-1$ and $d_i = |C_i|$ for $i=1,2$.

Let $u:= l.c.m.(d_1, d_2)/d_2$ and denote the vertices on the cycle $C_2$ and path $\gamma^+$ by
\begin{align*}
V(C_2)=f_1 f_2\cdots f_{d_2}, \\
V(\gamma^+)=B g_1 g_2 \cdots g_{\vert \gamma^+\vert-1} f_1,
\end{align*}
where $f_1=B'$.
We then construct a new graph $H$ from $G$ through the following steps:
\begin{enumerate}
\item[$(A)$] Let $H$ be the graph obtained from $G$ by deleting the cycle $C_2$;
\item[$(B)$] If $u> 1$, add to $H$ a simple path $\beta$ from $B$ to $B'$ such that
\begin{align*}
\vert\beta \vert &= \vert\gamma^+ \vert + (u-1) d_2,\\
V(\beta)&=Bg_1' g_2' \cdots g_{\vert \gamma^+\vert-1}' f_1^{(1)} f_2^{(1)}\cdots f_{d_2}^{(1)} \cdots \cdots f_1^{(u-1)} f_2^{(u-1)}\cdots f_{d_2}^{(u-1)} B';
\end{align*}
\item[$(C)$] For each $1\leq j\leq u-2$, add to $H$ an edge from $f_{d_2}^{(j)}$ to $B'$.
\end{enumerate}
See Figure \ref{G_and_H_single_spoke} for an example of $G$ and $H$ when $m=3$, $\vert C_1 \vert =4$ and $\vert C_2 \vert =3$.

\begin{figure}[]
\begin{center}
\begin{tikzpicture}[scale=0.6]
\draw [opacity=0] (0,1) grid (27,15);
\node [circle, draw, thick, inner sep=1.8pt] (GB) at (5,4) {$B$};
\node [circle, draw, thick, inner sep=1.8pt] (Gg1) at (4,7) {$g_1$};
\node [circle, draw, thick, inner sep=1.5pt] (Gf1) at (5,10) {$f_1$};
\node [circle, draw, thick, inner sep=3.5pt] (Ga) at (6,8) {$a$};
\node [circle, draw, thick, inner sep=3.0pt] (Gb) at (6,6) {$b$};
\node [circle, draw, thick, inner sep=3.0pt] (Gc) at (4,13) {$c$};
\node [circle, draw, thick, inner sep=3.0pt] (Gd) at (1,12) {$d$};
\node [circle, draw, thick, inner sep=3.0pt] (Ge) at (2,9) {$e$};
\node [circle, draw, thick, inner sep=3.0pt] (Gf2) at (9,10) {$f_2$};
\node [circle, draw, thick, inner sep=3.0pt] (Gf3) at (7,13) {$f_3$};
\draw [-stealth, black, thick] (GB) to [out=120, in=280] (Gg1);
\draw [-stealth, black, thick] (Gg1) to [out=90, in=240] (Gf1);
\draw [-stealth, black, thick] (Gf1) to [out=300, in=110] (Ga);
\draw [-stealth, black, thick] (Ga) to [out=280, in=80] (Gb);
\draw [-stealth, black, thick] (Gb) to [out=260, in=50] (GB);
\draw [-stealth, black, thick] (Gf1) to (Gc);
\draw [-stealth, black, thick] (Gc) to (Gd);
\draw [-stealth, black, thick] (Gd) to (Ge);
\draw [-stealth, black, thick] (Ge) to (Gf1);
\draw [-stealth, black, thick] (Gf1) to (Gf2);
\draw [-stealth, black, thick] (Gf2) to (Gf3);
\draw [-stealth, black, thick] (Gf3) to (Gf1);

\node [circle, draw, thick, inner sep=1.8pt] (HB) at (17,4) {$B$};
\node [circle, draw, thick, inner sep=1.8pt] (Hg1) at (16,7) {$g_1$};
\node [circle, draw, thick, inner sep=1.5pt] (Hf1) at (17,10) {$f_1$};
\node [circle, draw, thick, inner sep=3.5pt] (Ha) at (18,8) {$a$};
\node [circle, draw, thick, inner sep=3.0pt] (Hb) at (18,6) {$b$};
\node [circle, draw, thick, inner sep=3.0pt] (Hc) at (16,13) {$c$};
\node [circle, draw, thick, inner sep=3.0pt] (Hd) at (13,12) {$d$};
\node [circle, draw, thick, inner sep=3.0pt] (He) at (14,9) {$e$};
\node [circle, draw, thick, inner sep=0.5pt] (Hg11) at (19.5, 4.5) {$g_1^{(1)}$};
\node [circle, draw, thick, inner sep=0.5pt] (Hf11) at (22.2, 5.3) {$f_1^{(1)}$};
\node [circle, draw, thick, inner sep=0.5pt] (Hf21) at (24.5, 6.5) {$f_2^{(1)}$};
\node [circle, draw, thick, inner sep=0.5pt] (Hf31) at (25.5, 9) {$f_3^{(1)}$};
\node [circle, draw, thick, inner sep=0.5pt] (Hf12) at (25, 11.5) {$f_1^{(2)}$};
\node [circle, draw, thick, inner sep=0.5pt] (Hf22) at (23, 13.5) {$f_2^{(2)}$};
\node [circle, draw, thick, inner sep=0.5pt] (Hf32) at (20, 13.3) {$f_3^{(2)}$};
\node [circle, draw, thick, inner sep=0.5pt] (Hf13) at (19.5, 9.7) {$f_1^{(3)}$};
\node [circle, draw, thick, inner sep=0.5pt] (Hf23) at (22, 8.5) {$f_2^{(3)}$};
\node [circle, draw, thick, inner sep=0.5pt] (Hf33) at (22.5, 11.5) {$f_3^{(3)}$};
\draw [-stealth, black, thick] (HB) to [out=120, in=280] (Hg1);
\draw [-stealth, black, thick] (Hg1) to [out=90, in=240] (Hf1);
\draw [-stealth, black, thick] (Hf1) to [out=300, in=110] (Ha);
\draw [-stealth, black, thick] (Ha) to [out=280, in=80] (Hb);
\draw [-stealth, black, thick] (Hb) to [out=260, in=50] (HB);
\draw [-stealth, black, thick] (Hf1) to (Hc);
\draw [-stealth, black, thick] (Hc) to (Hd);
\draw [-stealth, black, thick] (Hd) to (He);
\draw [-stealth, black, thick] (He) to (Hf1);
\draw [-stealth, black, thick] (HB) to [out=0, in=200] (Hg11);
\draw [-stealth, black, thick] (Hg11) to [out=0, in=210] (Hf11);
\draw [-stealth, black, thick] (Hf11) to [out=10, in=220] (Hf21);
\draw [-stealth, black, thick] (Hf21) to [out=45, in=270] (Hf31);
\draw [-stealth, black, thick] (Hf31) to [out=90, in=290] (Hf12);
\draw [-stealth, black, thick] (Hf12) to [out=120, in=330] (Hf22);
\draw [-stealth, black, thick] (Hf22) to [out=170, in=20] (Hf32);
\draw [-stealth, black, thick] (Hf32) to [out=245, in=95] (Hf13);
\draw [-stealth, black, thick] (Hf13) to [out=320, in=170] (Hf23);
\draw [-stealth, black, thick] (Hf23) to [out=50, in=280] (Hf33);
\draw [-stealth, black, thick] (Hf33) to [out=160, in=60] (Hf1);
\draw [-stealth, black, thick] (Hf31) to [out=160, in=30] (Hf1);
\draw [-stealth, black, thick] (Hf32) to [out=190, in=80] (Hf1);

\node at (5, 1) [coordinate, draw, fill=black, label=above: The graph $G$] {};
\node at (19, 1) [coordinate, draw, fill=black, label=above: The graph $H$] {};

\end{tikzpicture}
\end{center}
\caption{An example of $G$ and $H$ with $m=3$, $\vert C_1 \vert=4, \vert C_2 \vert=3$.}
\label{G_and_H_single_spoke}
\end{figure}

We now construct a sliding block code $\psi:X_H \rightarrow X_G$ such that
$\psi$ is one-to-one and $\phi\circ\psi$ is finite-to-one and onto.  It will follow that $Z:= \psi(X_H$) is an SFT and $\phi|_Z$ is finite-to-one and onto.

Let $\Psi:  \mathcal{V}(H)\rightarrow \mathcal{V}(G)$ be the 1-block map defined by
\begin{enumerate}
\item[(a)] For any vertex $v$ on $\gamma^+, \gamma^-$ or $C_1$, $\Psi(v)=v$;
\item[(b)] For any $1\leq i\leq \vert \gamma^+\vert-1$, $\Psi(g_i')=g_i$; for any $1\leq j\leq d_2$ and $1\leq k\leq u-1$, $\Psi(f_{j}^{(k)})=f_j$.
\end{enumerate}
Let $\psi$ be the sliding block code induced by $\Psi$.
To show that $\psi$ is one-to-one it suffices to show that there exists some $M$ such that whenever $\psi(x) = y$, then $x_0$ can be uniquely determined from $y_{[-M,M]}$.  We show this  by considering the following possibilities for $y_0$:
\begin{enumerate}
\item[(1)] If $y_0$ is on $\gamma_-$ or  $C_1$ and $y_0\neq B'$,  then $x_0 = y_0$;
\item[(2)] If $y_0 = g_i$ for some $i$, let
$$N_1: =\min\{l\geq 0: y_l=g_{\vert \gamma^+\vert-1}\}\leq \vert \gamma^+ \vert-2.$$
Then $x_0 = g_i'$ if $y_{N_1+2}=f_2$ and $x_0=g_i$ otherwise;
\item[(3)] If $y_0 = f_j$ for some $j$, let
$$
N_2:= \min \{ \ell \ge 0: y_{-\ell} =  g_{|\gamma^+|}\} \leq (u-1)d_2.
$$
If $y_1\neq f_l$ for any $1\leq l\leq d_2$, then $x_0=f_1$; otherwise, $x_0  = f_j^{(k)}$ where $k = \lceil N_2/d_2 \rceil$.
\end{enumerate}
This shows that $\psi$ meets the criterion above to be one-to-one
with $M:= \max\{\vert \gamma^+ \vert, (u-1)d_2\}$.

Now we show that
$\phi \circ \psi: \widehat{X_H} \rightarrow Y$ is finite-to-one and onto. Note that by definition $\phi\circ \psi$ maps the central state $B$ of $H$ to $1$ and maps any other vertex to $0$.

To this end, 
first observe that any $k\in S$ must satisfy $k=m+s\cdot d_1 + t\cdot d_2$ for some $s, t\in \mathbb{Z}_{\geq 0}$. Noting from Lemma \ref{num_thy_lem} that there is a unique pair $(x, y)$ with $x,y\in \mathbb{Z}_{\geq 0}$ and $0\leq y<u$ such that $s\cdot d_1 + t \cdot d_2 = x \cdot d_1 + y \cdot  d_2$, we conclude that
\begin{align*}
(\Phi\circ\Psi)^{-1}(10^{k} 1)=
\begin{cases}\smallskip
V(\gamma^+(C_1)^x \gamma^-) \quad &\mbox{if $y=0$},\\
V(\beta(C_1)^x \gamma^- ) \quad &\mbox{if $0<y<u$.}
\end{cases}
\end{align*}
In particular, any block of the form $10^{k} 1$ with $k\in S$ has a unique preimage under $\Phi\circ \Psi$.
Similarly, one can show that
$|(\Phi\circ\Psi)^{-1}(0^\infty 1)| = 1$,
$|(\Phi\circ\Psi)^{-1}(1 0^\infty )| = u$,
and $|(\Phi\circ\Psi)^{-1}(0^\infty )| = d_1$.
Since each element $y \in Y$ is a concatenation of blocks of
the form $10^k, 0^\infty 1, 1 0^\infty$ and $0^\infty$ with $k\in S$,
$$
1 \le |(\phi \circ \psi)^{-1}(y)|  \le  \max(u, d_1).
$$
So  $\phi \circ \psi$ is finite-to-one and onto $Y$.  \qed

\begin{rem}
The subshift of finite type $Z$ is not unique: indeed, by interchanging the role of $C_1$ and $C_2$,
we can construct another SFT $Z'\subset \widehat{X_G}$ such that $\phi\vert_{Z'}$ is finite-to-one and onto $Y$.
\end{rem}


%



\end{document}